\documentclass[preprint,12pt]{elsarticle}

\usepackage{tikz}
\usetikzlibrary{shapes,arrows,calc}

\usepackage{enumitem}

\usepackage{amsmath,amssymb,ushort}

\usepackage[dvips]{epsfig}     
\usepackage{subfigure}         

\usepackage[unicode]{hyperref}
\hypersetup{breaklinks=true,colorlinks=true,linkcolor=blue,anchorcolor=blue,citecolor=blue,filecolor=blue,urlcolor=blue,bookmarksnumbered=true}

\usepackage{amsthm}

\usepackage{times} 

\theoremstyle{theorem}
\newtheorem{theorem}{Theorem}
\theoremstyle{corollary}
\newtheorem{corollary}{Corollary}
\theoremstyle{lemma}
\newtheorem{lemma}{Lemma}

\theoremstyle{assumption}
\newtheorem{assumption}{Assumption}
\theoremstyle{remark}
\newtheorem{remark}{Remark}
\theoremstyle{definition}
\newtheorem{definition}{Definition}
\theoremstyle{proposition}
\newtheorem{proposition}{Proposition}

\journal{Systems \& Control Letters}

\begin{document}

\begin{frontmatter}



\title{Invariantly admissible policy iteration for a class of \\ nonlinear optimal control problems}


\author[Yonsei Univ]{Jae Young Lee}\ead{jyounglee@yonsei.ac.kr}
\author[Yonsei Univ]{Jin Bae Park\corref{cor1}}\ead{jbpark@yonsei.ac.kr}
\author[Kyonggi Univ]{Yoon Ho Choi}\ead{yhchoi@kyonggi.ac.kr}
\address[Yonsei Univ]{Department of Electrical and Electronic Engineering, Yonsei University, 50 Yonsei-ro, Seodaemun-gu, Seoul, Korea}
\address[Kyonggi Univ]{Department of Electronic Engineering, Kyonggi University, 154-42 Gwanggyosan-ro, Yeongtong-gu, Suwon, Kyonggi-Do, Korea}
\cortext[cor1]{Corresponding author. Tel.: +82-2-2123-2773.}

\begin{abstract}
  In this paper, we propose
  a generalized successive approximation method (SAM),
  called invariantly admissible policy iteration (PI),
  for finding the solution
  to a class of input-affine nonlinear optimal control problems
  by iterations.
  Unlike the existing SAM,
  the proposed method updates the domain of the next policy and value function
  for admissibility (and invariance).
  In the existing SAM,
  the admissibility of the generated policies are guaranteed
  under the two implicit assumptions regarding Lyapunov's theorem and invariance,
  both of which are presented and discussed in this paper and are generally not true.
  On the contrary,
  the proposed invariantly admissible PI guarantees the admissibility in a more refined manner,
  without such assumptions.
  The admissibility and invariance of the updated region,
  with respect to the corresponding policies,
  are mathematically prove under the specific invariant admissible update rule.
  We also provide monotonic decreasing and uniform convergence properties of
  the sequence of value functions under certain conditions.
  Finally, numerical simulations are presented
  to illustrate the proposed PI method and its effectiveness.
\end{abstract}


\end{frontmatter}


\section{Introduction}

  In nonlinear optimal control problems,
  it is well-known
  that
  the optimal solution is
  directly related to the solution of the underlying Hamilton-Jacobi-Bellman (HJB) equation
  \cite{Bellman1957,Kirk2004,Lewis1995}.
  However,
  solving the HJB equation
  has been a formidable task until recently;
  hence,
  many of the numerical algorithms have been proposed
  for efficiently calculating the solution to the HJB equation
  \cite{Abu-Khalaf2005,Adhyaru2011,Alwardi2012,Beard1995,Beard1997,Cheng2007,Lee2012_WCCI,Lewis2009,Leake1967,Munos1999,Murray2002,Saridis1979,Spencer1996,Vrabie2009b,Wang2003}.

  Among such numerical algorithms,
  the successive approximation method (SAM)
  given in \cite{Beard1995,Beard1997,Leake1967,Saridis1979}
  has provided
  one basic idea of recursively solving the HJB equation.
  The algorithm starts with
  an initial admissible policy;
  during the recursions of the method,
  the agent finds
  the value function associated with the current policy
  (policy evaluation),
  and then the policy is updated using this associated value function
  (policy improvement).
  A class of algorithms using this idea
  is called policy iteration (PI),
  and
  many researchers
  have studied this idea
  in various ways
  and proposed
  their own algorithms
  from the perspectives of
  optimal control,
  adaptive (neuro-) dynamic programming,
  and reinforcement learning
  \cite{Abu-Khalaf2005,Adhyaru2011,Cheng2007,Lee2012_WCCI,Lewis2009,Murray2002,Vrabie2009b,Lee2012,Murray2003}.

  The PI method focused on in this paper is
  the SAM given by Beard, Saridis, and Wen \cite{Beard1995,Beard1997},
  which can be considered the infinite-horizon special case of the SAM given by Leake and Liu \cite{Leake1967},
  and becomes Newton method \cite{Kleinman1968}
  in the case of linear quadratic regulation (LQR).
  Note that many of the PI methods were also developed
  within the same optimal control framework as the SAM \cite{Beard1995,Beard1997,Saridis1979},
  and ideally
  all of them generate the same sequences of value functions
  and policies \cite{Lee2012_WCCI,Lewis2009,Murray2002,Vrabie2009b,Lee2012}.
  In other words,
  those PI algorithms
  can be considered the equivalents,
  and hence can be indirectly studied
  by analyzing the SAM of Beard et al. \cite{Beard1995,Beard1997}
  as a representative.

  The admissibility of the policies generated by the SAM \cite{Beard1995,Beard1997}
  is the motivation of this paper.
  Here, the admissibility of a policy roughly implies that
  the policy asymptotically stabilizes the system,
  and guarantees the finite value function on the domain of interest.
  In Theorem 5.3.1 in \cite{Beard1995},
  it was stated that
  the policies generated by the SAM \cite{Beard1995,Beard1997}
  are all admissible on the domain,
  and the sequence of the associated value functions is monotonically decreasing
  and converges to the optimal one,
  implying the improvement of the policy up to the optimal one.
  The proof was conducted based on Lemma 5.2.4 in \cite{Beard1995},
  which states the admissibility of the updated policy
  and the pointwisely monotonic decreasing property of the associated value functions.
  However,
  the related Lyapunov's theorem (Theorem 3.13 in \cite{Beard1995})
  used in its proof 
  for the infinite-horizon case
  implicitly assumed that
  \emph{the domain of the Lyapunov function
  is a subset of the stabilizing region}, and
  that \emph{the state trajectory generated by the nonlinear dynamics remains in that Lyapunov domain,
  so its existence is guaranteed for all future time}.
  The problem here is that
  \emph{both implicit assumptions on the Lyapunov domain are not true in general,
  as discussed in this paper}
  (see also Chapters 4.1, 4.2, and 8.2 in \cite{Khalil2002}, and Theorem 3.3 in \cite{Khalil2002}).
  To the best authors' knowledge,
  this problem does not happen only in the case of LQR
  since the stabilizing region becomes the entire $\mathbb{R}^n$-space
  and the state trajectory always exists for all time.

  To solve the aforementioned admissibility problem related to the nonlinear SAM \cite{Beard1995,Beard1997},
  this paper proposes a generalized SAM called invariantly admissible PI,
  which has an additional process to properly update the next admissible invariant region
  after each policy improvement step.
  For this,
  we refine and generalize
  the notion of an admissible policy
  given in \cite{Beard1995,Beard1997,Lewis2009}.
  Then, an invariantly admissible policy is precisely defined
  with detailed discussions
  on its necessity, the relevant Lyapunov's theorem,
  and the value functions for the underlying optimal control problem.
  From the discussions,
  a specific update rule for the invariantly admissible region in the proposed PI is presented.
  Without the aforementioned two implicit assumptions related to
  the Lyapunov's theorem
  (Theorem 3.13 in \cite{Beard1995}),
  it \emph{is} proven in this paper
  that
  the next region generated by the update rule
  is invariant and admissible for the current and next policies,
  and the sequence of corresponding value functions is monotonically decreasing.
  The conditions for convergence to the optimal solution
  are also provided with detailed discussions.
  Finally, numerical simulations are presented to illustrate the proposed PI method
  and its effectiveness.

  \section{Notations and mathematical terminology}
  $\mathbb{R}_+$ denotes the set of all nonnegative real numbers, \emph{i.e.},
  $\mathbb{R}_+ = [0, \infty)$;
  the set of all $n \times 1$ real vectors and $n \times m$ real matrices
  are denoted by $\mathbb{R}^{n}$ and $\mathbb{R}^{n \times m}$, respectively;
  $(\cdot)^T$ is the matrix transpose;
  $\|\cdot\|$ denotes a norm on a vector space $\mathbb{R}^n$.
  Throughout the paper,
  $\Omega$ (resp. $\bar \Omega$) denotes
  a subset of (resp. an invariant subset of)
  the given domain $\mathcal{D} \subseteq \mathbb{R}^n$ of the nonlinear dynamics.
  Here, the over-bar in $\bar \Omega$ means that it could be a compact set for some nice properties.
  The boundary of a subset $\Omega$ is denoted by $\partial \Omega$.
  All the mathematical notations
  including those given below
  will be clear and be precisely defined in this paper.

  \begin{description} \itemsep-0.3em
  \small
    \item [$\mathcal{A}(\Omega)$]: the set of all policies that are admissible on a subset $\Omega$;
    \item [$\mathcal{A}_{\mathcal{I}}(\Omega)$]: the set of all invariantly admissible policies on a subset $\Omega$;
    \item [$\mathcal{C}^0(\Omega)$]: the set of all continuous functions on a domain $\Omega$;
    \item [$\mathcal{C}^1(\Omega)$]: the set of all continuously differentiable functions on a domain $\Omega$;
    \item [$\bar B_0(r)$]: the closed ball in $\mathbb{R}^n$ with radius $r$. That is,
    $
        \bar B^0(r) := \{ x \in \mathbb{R}^n: \|x\| \leq r \};
    $
    \item [$R_A(\mu)$]: the region of attraction of the closed-loop system $\dot x = f(x) + g(x)\mu(x)$ in $\mathcal{D}$;    \\[-5pt]
    \item [$V$]: a Lyapunov function for an asymptotically stable closed-loop system;
    \item [$V^{\mu}$]: a value function for an (invariantly) admissible policy $\mu$;
    \item [$\nabla V^{\mu}$]: the gradient column vector of a value function $V^{\mu}$;
    \item [$V^*$]: the optimal value function;
    \item [$\mu$]: a policy $u = \mu(x)$ for the nonlinear system $\dot x = f(x) + g(x)u$;
    \item [$\mu^*$]: the optimal policy;
    \\[-5pt]
    \item [$\bar \Omega_c$]: the compact subset of a domain $\Omega$
                     defined with $V: \Omega \to \mathbb{R}_+$ by \eqref{eq:Omega_c};
    \item [$\bar \Omega_c^\mu$]: the compact subset of a domain $\Omega$ defined with $V^\mu: \Omega \to \mathbb{R}_+$ by \eqref{eq:Omega_c^mu};
    \item [$\Omega^*$]: the domain of $V^*$ on which $V^*$ is $\mathcal{C}^1$ and satisfies the HJB equation \eqref{eq:HJB}.
  \end{description}
  The notations related to the invariantly admissible PI
  are summarized as follows:
  \begin{description} \itemsep-0.3em
  \small
    \item [$\mu_i$]: the updated policy at $i$-th iteration;
    \item [$V^{\mu_i}$]: the value function for the policy $\mu$ obtained at $i$-th iteration;
    \item [$\hat V$]: the limit function to which $\{V^{\mu_i}\}$ converges;
    \item [$\bar \Omega_{c_i}^{\mu_i}$]: a compact set defined as
                                       $\bar \Omega_{c_i}^{\mu_i} := \{ x \in \mathbb{R}^n: V^{\mu_i}(x) \leq c_i \}$;
    \item [$\Omega_i$]: the updated region at $i$-th iteration such that $\Omega_i \subseteq \Omega_{i-1}$;
    \item [$\hat \Omega$]: the limit set of $\Omega_i$ defined as $\hat \Omega = \bigcap_{i=0}^\infty \Omega_i$;
    \item [$\mathcal{C}^1_\mathcal{A}(\hat \Omega)$]:
                        the set of all continuously differentiable value functions $V^\mu$
                        for $\mu \in \mathcal{A}({\hat \Omega})$.
  \end{description}

  \noindent
  \textbf{Terminology.}
  All the subsets in $\mathbb{R}^n$ (or in $\mathcal{D}$)
  presented in this paper
  are assumed to contain a neighborhood of the origin, and
  without loss of generality, have no isolated region or point from the origin.
  Using the above notations, a positive definite (resp. negative definite)
  function is precisely defined as

  \begin{definition}
    A function $V: \Psi \to \mathbb{R}_+$,
    where the domain $\Psi$ is a subset of $\mathbb{R}^{p}$ for some $p \in \{1,2,\cdots\}$
    containing a neighborhood of the origin,
    is said to be positive definite
    (resp. negative definite)
    on $\Psi$
    if and only if it is continuous on $\Psi$, $V(0) = 0$, and
    $V(x) > 0$ (resp. $V(x) < 0$) for all $x \in \Psi \setminus \{0\}$.
  \end{definition}

  \section{Preliminaries: invariant admissibility and nonlinear optimal control problems}

  In this paper,
  we consider the infinite-horizon nonlinear optimal control problem
  \eqref{eq:system}--\eqref{eq:performance measure}
  for the following continuous-time nonlinear system for time $t \in \mathbb{R}_+$
  \begin{equation}
    \dot x(t) = f(x(t)) + g(x(t))u(x(t)), \;\; x(0) = x_0 \in \mathcal{D} \subseteq \mathbb{R}^n,
    \label{eq:system}
  \end{equation}
  \begin{equation*}
  \textrm{where }
  \begin{cases}
        x: \mathbb{R}_+ \to \mathbb{R}^n: \textrm{the system state for time $t \in \mathbb{R}_+$}; \\
        u: \mathbb{R}^n \to \mathbb{R}^m: \textrm{the control input function governed by a given control (or } \\
        \qquad\qquad\qquad\;\; \textrm{policy) $\mu(x)$ (see Definition \ref{def:policy});}\\
        \mathcal{D} \subseteq \mathbb{R}^n:  \textrm{the domain of $f$ and $g$ containing a neighborhood of the origin}; \\
        f: \mathcal{D} \to \mathbb{R}^n: \textrm{a given locally Lipshitz continuous nonlinear function that} \\
        \qquad\qquad\quad\;\;\; \textrm{satisfies $f(0) = 0$}; \\
        g: \mathcal{D} \to \mathbb{R}^{n \times m}: \textrm{a given locally Lipshitz continuous nonlinear function}, \\
  \end{cases}
  \end{equation*}
  and the performance measure
  \begin{equation}
    J(x_0, u(\cdot))
    = \int_0^\infty r(\phi(\tau; x_0, u),u(\phi(\tau; x_0, u))) \; d\tau,
    \label{eq:performance measure}
  \end{equation}
  where
  \begin{description}
    \item [{\;\;\small $\bullet$} $\phi(\tau; x_0, u):$]
    the state trajectory $x(\tau)$ at time $\tau \in \mathbb{R}_+$ generated by \eqref{eq:system} with
    the initial condition $x_0 \in \mathcal{D}$ and a given policy $u = \mu(x)$;
    \\[-10pt]
    \item [{\;\;\small $\bullet$} $r: \mathcal{D} \times \mathbb{R}^{m} \to \mathbb{R}_+:$]
        the given positive definite cost function
        on $\mathcal{D} \times \mathbb{R}^{m}$ defined as
        \[
            r(x,u) := Q(x)+u^TRu
        \]
        for a positive definite function $Q: \mathcal{D} \to \mathbb{R}_+$ on $\mathcal{D}$
        and a positive definite matrix $R \in \mathbb{R}^{m \times m}$.
  \end{description}
  Here, the notion of a policy $\mu(x)$ for the system \eqref{eq:system}
  is precisely defined as follows.

  \begin{definition}
  \label{def:policy}
    A function $\mu: \mathcal{D} \to \mathbb{R}^m$ is said to be
    a policy on a subset $\Omega \subseteq \mathcal{D}$
    if and only if $\mu$ is continuous on $\Omega$ and satisfies $\mu(0) = 0$.
  \end{definition}

  Note that
  the nonlinear dynamics \eqref{eq:system},
  which has the origin `0' as an equilibrium,
  can be regarded as
  the general description
  of the systems such as
  feedback linearizable systems \cite{Khalil2002}, strict feedback systems \cite{Khalil2002},
  bilinear systems \cite{Beard1995}, and
  many practical nonlinear systems \cite{Beard1995},
  all of which can be stabilized
  by a continuous feedback control $u = \mu(x)$
  for the equilibrium `$0$'.
  For the existence of the solution $\phi(t; x_0, \mu)$ of
  the nonlinear dynamics \eqref{eq:system} $\forall t \geq 0$,
  we assume that
  \begin{assumption}
  \label{assumption:Lipschitz}
    For any given policy $\mu(x)$, $f(x) + g(x)\mu(x)$ is locally Lipschitz continuous on the domain $\mathcal{D}$.
  \end{assumption}
  If the policy $\mu(x)$ is continuously differentiable on $\mathcal{D}$, \emph{i.e.}, $\mu \in \mathcal{C}^1(\mathcal{D})$,
  then it can be easily shown that $f(x) + g(x)\mu(x)$ is locally Lipschitz continuosu on $\mathcal{D}$,
  so Assumption \ref{assumption:Lipschitz} holds.
  In this paper, Assumption \ref{assumption:Lipschitz} suffices for the analysis, and
  we do not assume such a strict differentiability assumption on $\mu(x)$.
  Next,
  we precisely define a feasible trajectory and a stabilizing policy
  on a given subset $\Omega$ of $\mathcal{D}$.

  \begin{definition} [\textbf{Feasible trajectory}]
  \label{def:feasible trj}
  For a given policy $\mu(x)$,
  the state trajectory $\phi(t; x_0, \mu)$
  is said to be feasible on a subset $\Omega \subseteq \mathcal{D}$
  if and only if
  \begin{equation}
     x_0 \in \Omega \textrm{ implies } \phi(t; x_0, \mu) \in \mathcal{D} \;\; \textrm{for all } t \geq 0.
     \label{eq:well-define trj}
  \end{equation}
  \end{definition}

  \begin{definition} [\textbf{Stabilizing policy}]
  \label{def:stabilizing policy}
  A policy $\mu(x)$ is said to asymptotically stabilize the system $(f,g)$ on
  $\Omega \subseteq \mathcal{D}$ (or stabilizing on $\Omega$)
  if and only if
  \begin{enumerate} \itemsep-0.3em
  \item
  $\phi(t; x_0, \mu)$ exists $\forall x_0 \in \Omega$ and $\forall t \geq 0$;
  \item
  the equilibrium `0' of the resulting closed-loop system $\dot x = f(x) + g(x)\mu(x)$
  is stable;
  \item
  $\lim_{t \to \infty} \phi(t; x_0, \mu) = 0$ for all $x_0 \in \Omega$.
  \end{enumerate}
  \end{definition}

  For a given stabilizing policy $\mu(x)$,
  the region of attraction of the closed-loop system $\dot x = f(x)+g(x)\mu(x)$ is defined as
  \[
    R_A(\mu) :=
    \big \{
      x_0 \in \mathcal{D}: \phi(t; x_0, \mu) \to 0 \textrm{ as } t \to \infty
    \big \};
  \]
  Similarly, we define the value function $V^\mu(x_0)$ for $x_0 \in \mathcal{D}$,
  if it exists, as
  \begin{equation*}
    V^\mu(x_0) := J(x_0, u(\cdot))|_{u = \mu(x)}.
  \end{equation*}
  Since $Q(0) = 0$, $\mu(0) = 0$,
  and $\phi(t;x_0, \mu)|_{x_0 = 0} = 0$ for all $t \geq 0$,
  we have $V^\mu(0) = 0$.
  So, by the positive definiteness of $r(x,u)$ on $\mathcal{D} \times \mathbb{R}^m$,
  $V^\mu$ is always positive definite on its domain.
  Using Definitions
  \ref{def:policy}--\ref{def:stabilizing policy},
  the notion of an admissible policy
  given by Beard et al. \cite{Beard1997} for the existence of $V^\mu$
  can be re-defined in a refined, generalized manner as follows.

  \begin{definition} [\textbf{Admissible policy}]
  \label{def:admissibility}
    A policy $\mu(x)$
    is admissible on a subset $\Omega \subseteq \mathcal{D}$,
    denoted by $\mu \in \mathcal{A}(\Omega)$,
    if and only if
    \begin{enumerate}  \itemsep-0.3em
     \item
     $\mu(x)$ asymptotically stabilizes the system $(f,g)$ on $\Omega$;
     \item
     $\phi(t; x_0, \mu)$ is feasible on $\Omega$;
     \item
     $V^\mu(x_0) < \infty$, \; $\forall x_0 \in \Omega$.
    \end{enumerate}
  \end{definition}

    For the nonlinear dynamics \eqref{eq:system},
    we assume the existence of an admissible policy.
  \begin{assumption}
    There exist a policy $\mu(x)$ and a subset $\Omega \subseteq \mathcal{D}$ for the nonlinear system \eqref{eq:system}
    such that $\mu \in \mathcal{A}(\Omega)$.
  \end{assumption}

  Note that $\mu \in \mathcal{A}(\Omega)$
  implies that $\mu$ is stabilizing on $\Omega$,
  and thereby, $\Omega \subseteq R_A(\mu)$.
  Compared with \cite{Beard1997},
  the concept of admissibility in Definition \ref{def:admissibility} is
  refined and slightly generalized.
  First, it is defined on a subset $\Omega$ of $\mathcal{D}$,
  so contains the previous definition as a special case ``$\Omega = \mathcal{D}$'' \cite{Beard1997};
  second, we assume
  that $\phi(t; x_0, \mu)$ $(t \geq 0)$ is feasible on $\Omega$,
  so $\phi(t; x_0, \mu)$ remains in the domain $\mathcal{D}$
  for all $t \geq 0$ and all $x_0 \in \Omega$.
  This condition is guaranteed
  if $\Omega \subseteq R_A(\mu)$ is satisfied and
  $\mathcal{D}$ contains either
  $R_A(\mu)$ or its invariant subset containing $\Omega$.
  However, such a domain $\mathcal{D}$ is hard to determine (or even impossible)
  unless $\mathcal{D} = \mathbb{R}^n$
  since both $R_A(\mu)$ and its invariant subset depend on the policy $\mu$,
  and hence so does the determination of $\mathcal{D}$.
  Therefore,
  instead of imposing such an unrealistic assumption on $\mathcal{D}$,
  we introduce the concept of invariant admissibility as follows.

  \begin{definition} [\textbf{Invariantly admissible policy}]
  \label{def:invariant admissibility}
    A policy $\mu(x)$
    is invariantly admissible on a subset $\bar \Omega \subseteq \mathcal{D}$
    containing a neighborhood of the origin,
    denoted by $\mu \in \mathcal{A}_\mathcal{I}(\bar \Omega)$,
    if and only if
    \begin{enumerate} \itemsep-0.3em
     \item
     $\mu \in \mathcal{A}(\bar \Omega)$;
     \item
     $\bar \Omega$ is invariant under the policy $\mu$, \emph{i.e.},
     \begin{equation}
      \textrm{if } x_0 \in \bar \Omega, \textrm{ then } \phi(t;x_0, \mu) \in  \bar \Omega \,\;
      \textrm{ for all } t \geq 0.
      \label{eq:invariance trj}
     \end{equation}
    \end{enumerate}
  \end{definition}

  \begin{proposition}
  \label{prop:A to AI}
    $\mu \in \mathcal{A}_\mathcal{I}(\bar \Omega)$ implies
    $\mu \in \mathcal{A}(\bar \Omega)$.
  \end{proposition}

  Note that
  the invariance condition \eqref{eq:invariance trj}
  in Definition \ref{def:invariant admissibility}
  replaces
  the feasibility condition \eqref{eq:well-define trj}
  in Definition \ref{def:feasible trj}.
  By Theorem 3.3 in \cite{Khalil2002} and
  Assumption \ref{assumption:Lipschitz},
  the invariance \eqref{eq:invariance trj}
  also guarantees the existence of
  the unique solution $\phi(t; x_0, \mu)$
  for all $x_0 \in \bar \Omega$
  and all $t \geq 0$
  if $\bar \Omega$ is compact.
  Related to these observations and invariant admissibility,
  we look inside
  a variant of the (local) Lyapunov's theorem for asymptotic stability
  (Theorem 4.1 in \cite{Khalil2002})
  on a compact set $\bar \Omega_c$ defined as
  \begin{equation}
  \bar \Omega_c = \{x \in \mathbb{R}^n: V(x) \leq c\},
  \label{eq:Omega_c}
  \end{equation}
  where $V: \Omega \to \mathbb{R}_+$ is a Lyapunov function for an asymptotically stable closed-loop system $\dot x = f(x) + g(x)\mu(x)$
  on a domain $\Omega \subseteq \mathcal{D}$,
  and $c$ is a constant determined in such a way that
  $\bar \Omega_c$ is contained by $\Omega$, \emph{i.e.},
  $\bar \Omega_c \subseteq \Omega$.
  For the proof, see Theorem 4.1 in \cite{Khalil2002} and its proof.

  \begin{theorem} \label{thm:Lyapunov thm}
    For a subset $\Omega \subseteq \mathcal{D}$,
    if there exists a function $V: \Omega \to \mathbb{R}_+$
    such that $V$ is positive definite on $\Omega$,
    $V \in \mathcal{C}^1(\Omega)$, and
    $\dot V \equiv ( \partial V / \partial x )^T (f+g\mu)$ is negative definite on $\Omega$,
    then,
    \begin{enumerate}
    \item
    $\mu(x)$ asymptotically stabilizes the system $(f,g)$ on $\bar \Omega_c$;
    \item
    $x_0 \in \bar \Omega_c$
    implies $\phi(t; x_0, \mu) \in \bar \Omega_c$ $\forall t \geq 0$.
    \end{enumerate}
  \end{theorem}

  Theorem \ref{thm:Lyapunov thm}
  provides
  an asymptotically stable invariant region $\bar \Omega_c$,
  which is a compact invariant subset of $R_A(\mu)$.
  On this invariant region $\bar \Omega_c$,
  existence and feasibility of
  the unique solution $\phi(t;x_0,\mu)$ are guaranteed.
  Therefore,
  if $V^\mu(x) < \infty$ holds for all $x \in \bar \Omega_c$,
  then the conditions in Theorem \ref{thm:Lyapunov thm}
  imply $\mu \in \mathcal{A}_\mathcal{I}(\bar \Omega_c)$.

  \begin{remark}
  \label{remark:domain of Lyapunov function}
  $\Omega \setminus \bar \Omega_c$ may not be
  a stabilizing region
  since $\Omega \setminus \bar \Omega_c \subseteq R_A(\mu)$ is not guaranteed.
  So, $\phi(t; x_0, \mu)$ for some $x_0 \in \Omega \setminus \bar \Omega_c$
  may leave the domain $\mathcal{D}$ and even may diverge to $\infty$.
  In this situation,
  $\dot V(x) < 0$ and even the existence of $\phi(t; x_0, \mu)$ ($t \geq 0$)
  are not guaranteed (see Section 8.2 in \cite{Khalil2002} for more discussions).
  \end{remark}

  If $V^{\mu} \in \mathcal{C}^1(\Omega)$, then
  it satisfies the Lyapunov equation
  for the system \eqref{eq:system}:
 \begin{align}
    \nabla^T V^{\mu}(x) \cdot \big ( f(x) + g(x) \mu(x) \big ) = -r(x, \mu(x)), \;\; \forall x \in  \Omega,
   \label{eq:Lyapunov equation for mu}
 \end{align}
 which is
 the infinitesimal version of \eqref{eq:performance measure}
 and implies $\dot V^\mu(x) = - r(x,\mu) < 0$
 along the trajectory $\phi(t; x_0, \mu)$.
 In this case, since $V^{\mu}$ is positive definite on its domain,
 \eqref{eq:Lyapunov equation for mu} guarantees
 that $V^{\mu}$ is a Lyapunov function for the closed-loop system $\dot x = f + g\mu$ satisfying the conditions in Theorem \ref{thm:Lyapunov thm}.
 This provides the following converse lemma of Proposition \ref{prop:A to AI}
 on a compact subset
 $\bar \Omega_c^\mu$ of $ \Omega$ defined similarly to $\bar \Omega_c$ by
  \begin{equation}
  \label{eq:Omega_c^mu}
    \bar \Omega_c^\mu := \{x \in \mathbb{R}^n: V^\mu(x) \leq c\}, 
  \end{equation}
  where $c > 0$ is chosen such that $\bar \Omega_c^\mu$ is contained by
  the domain $\Omega \subseteq \mathcal{D}$ of $V^\mu$,
  \emph{i.e.}, $\bar \Omega_c^\mu \subseteq  \Omega \subseteq \mathcal{D}$.
  The proof can be easily done by
  applying Theorem \ref{thm:Lyapunov thm} with the Lyapunov function $V = V^\mu$
  satisfying \eqref{eq:Lyapunov equation for mu} $\forall x \in \Omega$.

  \begin{lemma}
  \label{lemma:admissibility to inv. admissibility}
    If $\mu \in \mathcal{A}(\Omega)$ and $V^\mu \in \mathcal{C}^1(\Omega)$,
    then $\mu \in \mathcal{A}_\mathcal{I}(\bar \Omega_c^\mu)$.
  \end{lemma}

 \begin{remark}
 The existence of the unique $\mathcal{C}^1$ value function $V^\mu$
 on a subset of $R_A(\mu)$ is guaranteed under certain conditions,
 for example, if:
 \begin{enumerate}
   \item
   $Q(x)$ has second partial derivatives that are continuous,
   and all the real parts of the eigenvalues of $\nabla (f+g\mu)|_{x=0}$ are negative
   \cite{Knobloch1974} (see also Theorem 3 in \cite{Kaslik2005});
   \item
   the functions $\mu(x)$ and $Q(x)$ are continuously differentiable, and $\mu$ is admissible
   on the domain (Lemma 3.1.6 in \cite{Beard1995});
   \item
   the functions $f(x)$, $g(x)$, and $Q(x)$ are all smooth on the domain,
   and all the real parts of the eigenvalues of $\nabla f(x)|_{x=0}$ are negative
   \cite{Murray2002,Murray2003}.
 \end{enumerate}
 \end{remark}

 The next lemma is
 a refined, generalized version of Lemma 3.1.9 in \cite{Beard1995},
 and states that
 the admissibility is preserved
 in a feasible stabilizing region.

 \begin{lemma}
    \label{lemma:admissibility on Upsilon}
    Assume $\mu \in \mathcal{A}(\Omega)$
    for a subset $\Omega \subseteq \mathcal{D}$
    containing a neighborhood of the origin.
    Let $\Upsilon \subseteq \mathcal{D}$ be a feasible subset of $R_A(\mu)$.
    Then, $V^\mu$ is defined for all $x \in \Upsilon$, and $\mu$ is admissible on $\Upsilon$, i.e., $\mu \in \mathcal{A}(\Upsilon)$.
 \end{lemma}

 \begin{proof}
    Let $N_0$ be the neighborhood of the origin
    contained by $\Omega$.
    Then, $\mu \in \mathcal{A}(\Omega)$ implies
    $\mu \in \mathcal{A}(N_0)$,
    so $V^{\mu}(x) < \infty$ for all $x \in N_0$.
    Since $\Upsilon \subseteq \mathcal{D}$
    is a feasible subset of $R_A(\mu)$, we have
    ``$
        x_0 \in \Upsilon \textrm{ implies } \phi(t;x_0, \mu)  \in \mathcal{D}
    $ for all $t \geq 0$''
    and ``$\lim_{t \to \infty} \phi(t;x_0, \mu) =0$ for all $x_0 \in \Upsilon$''.
    Here, the latter implies that
    there is a time $T > 0$ such that ``$\phi(T;x_0, \mu) \in N_0$''.
    Therefore, we have $\mu \in \mathcal{A}(\Upsilon)$
    since from \eqref{eq:performance measure} and the definition of $V^\mu$,
    \begin{align*}
        &V^{\mu}(x_0)  \\
         &= \int_0^T r(\phi(\tau; x_0, \mu),\mu(\phi(\tau; x_0, \mu))) \; d\tau
         + \underbrace{\int_T^\infty r(\phi(\tau;x_0, \mu),\mu(\phi(\tau;x_0, \mu))}_{=V^{\mu} (\phi(T; \,x_0, \mu))} \; d\tau\\
         &< \;\; V^{\mu} (\phi(T; \,x_0, \mu)) < \infty
    \end{align*}
    holds for all $x_0 \in \Upsilon$.
  \end{proof}

  Define the Hamiltonian $H:\mathbb{R}^n \times \mathbb{R}^m \times \mathbb{R}^n \to \mathbb{R}$
  for the nonlinear optimal control problem \eqref{eq:system}--\eqref{eq:performance measure} as
  \[
     H(x,u,p) := r(x,u) + p^T (f(x) + g(x)u).
  \]
  Then, the Lyapunov equation \eqref{eq:Lyapunov equation for mu}
  can be represented as $H(x,\mu,\nabla V^\mu) = 0$,
  and minimizing $H(x,\mu,\nabla V^*)$ among all admissible policies $\mu$
  yields the optimal policy $\mu^*(x)$ below:
  \begin{align}
    \mu^{*}(x)  = -\frac{1}{2}R^{-1}g^{T}(x)\nabla V^{*}(x),
    \label{eq:optimal policy u*}
  \end{align}
  where $V^{*}$ is
  the optimal value function defined
  as $V^* := V^{\mu^*}$.
  Furthermore, substituting \eqref{eq:optimal policy u*}
  into \eqref{eq:Lyapunov equation for mu} and rearranging the equation yields the well-known HJB equation:
  \begin{align}
    0 = Q(x) + \nabla {V^*}^Tf(x)
       - \frac{1}{4} \nabla {V^*}^T g(x) R^{-1} g^T(x) \nabla V^*.
       \label{eq:HJB}
 \end{align}
 For the optimal solution $V^*$, we assume throughout the paper that
 \begin{assumption}
    \label{assumption:HJB}
    $V^*$ is the unique $\mathcal{C}^1$-positive definite solution of the HJB equation \eqref{eq:HJB}
    on a subset $\Omega^* \subseteq \mathcal{D}$.
 \end{assumption}

\section{Policy iteration with admissible region update: invariantly admissible PI}
\label{section:invariant admissibility of PI}

 In this section,
 we focus on and discuss the invariant admissibility
 of the SAM \cite{Beard1995,Beard1997}.
 Then, the advanced algorithm, called invariantly admissible PI in this paper, is proposed
 which determines not only the value function and the next policy but its invariant admissible region
 at each iteration.

 In Lemma 5.2.4 in \cite{Beard1995},
 it was stated that the policies $\mu_i$'s
 generated by the SAM \cite{Beard1995,Beard1997}
 with an initial admissible policy $\mu_0 \in \mathcal{A}(\Omega)$
 are all admissible on $\Omega$.
 In the proof of the lemma,
 however,
 \begin{enumerate}
 \item [(1)]
 Lyapunov's theorem (Theorem 3.13 in \cite{Beard1995}) was applied
 under the implicit assumption that
 the domain $\Omega \subseteq \mathcal{D}$ of
 the Lyapunov function $V^{\mu_i}:\Omega \to \mathbb{R}_+$
 for the $i$-th admissible policy $\mu_i$
 is a subset of $R_A(\mu_i)$;
 \end{enumerate}
 the assumption is not true in general
 as mentioned in Remark \ref{remark:domain of Lyapunov function}.
 Moreover,
 \begin{enumerate}
 \item [(2)]
 the domain $\mathcal{D}$,
 which was equal to $\Omega$ in \cite{Beard1995,Beard1997},
 was arbitrarily given, not as an invariant estimate of $R_A(\mu_i)$,
 so that the trajectory $\phi(t; x_0, \mu_i)$ starting in $\mathcal{D}$
 may escape the domain $\mathcal{D}$ and may not be feasible.
 \end{enumerate}

\tikzstyle{io} = [trapezium,trapezium left angle=75,trapezium right angle=-75,minimum height=0.6cm, draw, fill=white,
text width=20em, inner sep=5pt]
\tikzstyle{decision} = [diamond, shape aspect=4, draw, fill=white!20,
text width=13em, text badly centered, inner sep=0pt]
\tikzstyle{block} = [rectangle, draw, fill=white!20,
text width=27em, minimum height=4em, inner sep= 5pt, -latex']
\tikzstyle{line} = [draw, -latex']
\tikzstyle{cloud} = [draw, ellipse, fill=blue!20, minimum height=1em]

\begin{figure}

    \centering
    \begin{tikzpicture} [node distance = 2cm, auto, execute at begin node=\normalsize]
    \node [io] (init)
        {$\,$\\
         \textbf{Step 1: Initialization}\\[3pt]
         Set $i = 0$ and $\varepsilon > 0$ be a small constant.
         Let $\mu_0$ be a given policy that is admissible on a compact subset $\Omega_0 \subseteq \mathcal{D}$, \emph{i.e.},
         $\mu_0 \in \mathcal{A}(\Omega_0)$.\\[-10pt]
         $\,$
        };

    \node [cloud, left of=init, node distance=6.5cm] (start) {Start};

    \node [block, below of=init, node distance = 3.4cm] (policy evaluation)
        {
        $\,$\\
        \textbf{Step 2: Policy Evaluation}\\[3pt]
         Find the value function $V^{\mu_i} \in \mathcal{C}^1(\Omega_i)$ for $\mu_i \in \mathcal{A}(\Omega_i)$ such that
         \begin{align}
            (\nabla V^{\mu_i}(x))^T \big ( f(x) + g(x) \mu_i(x) \big ) = -r(x, \mu_i(x)),  \forall x \in \Omega_i.
            \label{eq:Policy Evaluation}
            \\[-10pt] \notag
         \end{align}
         };

    \node [block, below of=policy evaluation, node distance = 3.7cm] (policy improvement)
        {
        $\,$\\
        \textbf{Step 3: Policy Improvement}\\[3pt]
         Fine the next policy $\mu_{i+1}: \mathcal{D} \to \mathbb{R}^n$ whose restriction on $\Omega_i$ satisfies $\forall x \in \Omega_i$,
         \\[-15pt]
         \begin{align}
         \mu_{i+1}(x) = - \frac{1}{2}R^{-1} \, g^T(x) \nabla V^{\mu_{i}}(x).
         \label{eq:Policy Improvement}
         \\[-12pt] \notag
         \end{align}
         };

    \node [block, below of=policy improvement, node distance = 3.7cm] (region update)
        {
        $\,$\\
        \textbf{Step 4: Invariant Admissible Region Update}\\[3pt]
         Find the next compact region $\Omega_{i+1} \subseteq \mathcal{D}$ such that
         \begin{align}
         \mu_{i} \in \mathcal{A}_\mathcal{I}(\Omega_{i+1})
         \textrm{ and }
         \mu_{i+1} \in \mathcal{A}_\mathcal{I}( \Omega_{i+1}).
         \notag
         \\[-10pt] \notag
         \end{align}
         };

    \node [decision, below of=region update, node distance = 3.4cm] (converge)
        {\textbf{Step 5: Convergence}\\[3pt]
        $
         \textrm{Is } \sup_{x \in \Omega_{i+1}} \|\mu_{i+1} - \mu_i\| < \varepsilon?
        $
        };

    \node [left of=policy improvement, node distance=4cm]
        (again) {};

    \node [cloud, below of=converge, node distance=2.5cm] (end) {\;\;End\;\;};

    \path [line] (start) -- (init);
    \path [line] (init) -- (policy evaluation);
    \path [line] (policy evaluation) -- (policy improvement);
    \path [line] (policy improvement) -- (region update);
    \path [line] (region update) -- (converge);
    \path [line] (converge) -| node[below, xshift=1.5cm]{no}
                                ([xshift=-1.5cm, yshift=-0.5cm]policy improvement.south west)
                            |-  (policy evaluation);
    \path [line] (converge) -| node[above, xshift=1.5cm]{$i \leftarrow i+1$}
                                ([xshift=-1.5cm, yshift=-0.5cm]policy improvement.south west)
                            |-  (policy evaluation);
    \path [line] (converge) -- node[right]{yes}(end);
    \end{tikzpicture}
    \caption{The proposed invariantly admissible PI algorithm}
    \label{fig:modified PI}
\end{figure}
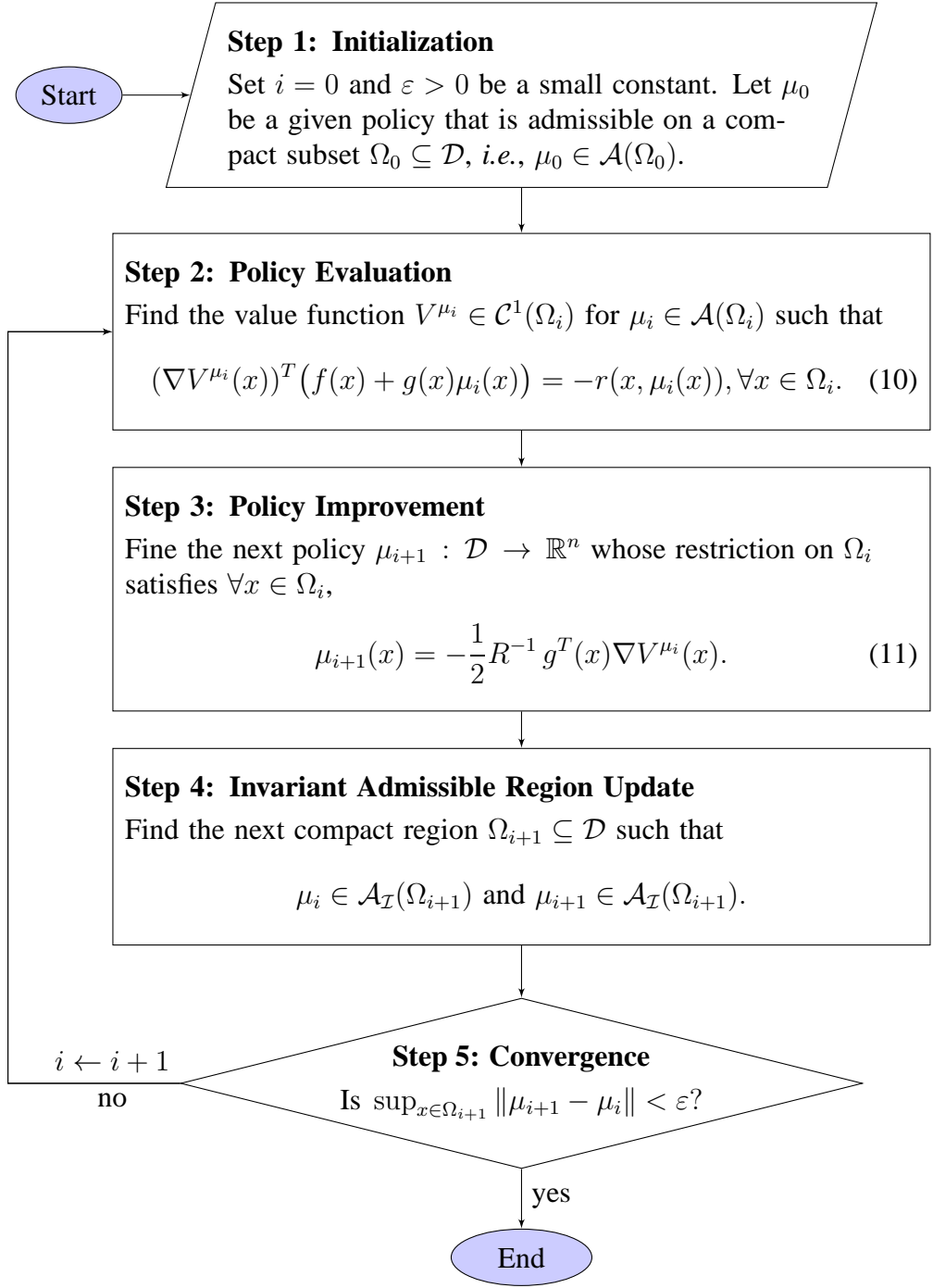

 These two problems can be solved at the same time
 if the domain $\mathcal{D}$
 is given
 as an invariant estimate of the regions of attraction $R_A(\mu_i)$
 for all the closed-loop systems $\dot x = f+g\mu_i$.
 That is, $\forall i \in \mathbb{Z}_+$, $\mathcal{D} \subseteq R_A(\mu_i)$ and
 \[
  x_0 \in \mathcal{D} \textrm{ implies }
  \phi(t;x_0, \mu_i) \in \mathcal{D}, \;\; \forall t \geq 0.
 \]
 To determine such an invariant attraction domain $\mathcal{D}$,
 however,
 the knowledge about
 all the updated policies $\mu_i$ ($i = 0,1,2, \cdots$) in PI
 has to be given \emph{a priori}, 
 which is impossible but $i = 0$ before the algorithm runs.

 Instead of this unrealistic approach,
 this paper solves the addressed problems
 by using another technique,
 which is used in the proposed PI method
 and determines, for a given domain $\mathcal{D}$ and at each $i$-th iteration,
 the next 
 region $ \Omega_{i+1}$
 such that
 both the current policy $\mu_i$ and the next policy $\mu_{i+1}$ are invariantly admissible on $ \Omega_{i+1}$,
 \emph{i.e.},
 $\mu_i, \mu_{i+1} \in \mathcal{A}_\mathcal{I}( \Omega_{i+1})$.
 Fig. \ref{fig:modified PI} describes
 the whole process of the proposed PI algorithm,
 where the next invariant admissible domain $ \Omega_{i+1}$ is determined
 in the process of ``invariant admissible region update''
 that is newly introduced for the safe learning of
 both the optimal solution $(V^*, \mu^*)$ and the corresponding invariant admissible region.
 Policy evaluation and improvement are the same as those
 in the SAM \cite{Beard1995,Beard1997}
 except that they are performed
 in the domain $ \Omega_i$, instead of
 in the whole domain $\mathcal{D}$.

 \subsection{Invariant admissibility and monotonic decreasing properties}
 Related to policy evaluation and policy improvement,
 the following theorem states
 the invariant admissibility of the policies and the  monotonic decreasing property
 of the sequence of associated value functions on the compact subset $\bar \Omega_{c_i}^{\mu_i} \subseteq \mathcal{D}$
 defined with a positive constant $c_i > 0$ as
 \[
  \bar \Omega_{c_i}^{\mu_i} = \{x \in \mathbb{R}^n : V^{\mu_i}(x) \leq c_i \}.
 \]

 \begin{theorem}
   \label{thm:main thm1}
    Assume $\mu_i \in \mathcal{A}(\Omega_i)$, and
    let $\Upsilon_i \subseteq \mathcal{D}$ be any feasible subset of $R_A(\mu_i)$.
    Then, $\mu_i \in \mathcal{A}(\Upsilon_i)$.
    Moreover,
    if $V^{\mu_i}$ is continuously differentiable on $\Upsilon_i$,
    $\mu_{i+1}$ satisfies \eqref{eq:Policy Improvement} for all $x \in \Upsilon_i$,
    and $c_i$ is chosen such that $\bar \Omega_{c_i}^{\mu_i} \subseteq \Upsilon_i$,
    then
    \begin{enumerate} \itemsep-0.3em
    \item
    $\mu_{i+1}$ is a policy on $\bar \Omega_{c_i}^{\mu_i}$;
    \item
    $\mu_i, \mu_{i+1} \in \mathcal{A}_\mathcal{I}( \bar \Omega_{c_i}^{\mu_i})$;
    \item
    for all $x \in  \bar \Omega_{c_i}^{\mu_i}$,
    the next value function $V^{\mu_{i+1}}$ satisfies
   \begin{equation}
      0 < V^{\mu_{i+1}}(x) \leq V^{\mu_{i}}(x) < \infty.
      \label{eq:lemma:decreasing property of V}
   \end{equation}
   \end{enumerate}
 \end{theorem}

 \begin{proof}
   First, $\mu_i \in \mathcal{A}(\Upsilon_i)$ is easily proven
   by applying Lemma \ref{lemma:admissibility on Upsilon}
   with $\Omega = \Omega_i$ and $\Upsilon = \Upsilon_i$.
   Here, $\mu_i \in \mathcal{A}(\Upsilon_i)$ implies that
   $V^{\mu_i}(x)$ is finite $\forall x \in \Upsilon_i$.
   Next, assume $V^{\mu_i}$ is continuously differentiable on $\Upsilon_i$,
   and $\bar \Omega_{c_i}^{\mu_i} \subseteq \Upsilon_i$.
   Then, we have $\mu_i \in \mathcal{A}_\mathcal{I}(\bar \Omega_{c_i}^{\mu_i})$
   by Lemma \ref{lemma:admissibility to inv. admissibility}
   with
    $\Omega = \Upsilon_i$ and
    $\bar \Omega_{c}^{\mu} = \bar \Omega_{c_i}^{\mu_i}$.
   For the remaining of the proof,
   assume further that
   $\mu_{i+1}$ satisfies \eqref{eq:Policy Improvement} for all $x \in \Upsilon_i$.

   We now show that $\mu_{i+1}$ is a policy on $\bar \Omega_{c_i}^{\mu_i}$.
   Since    $V^{\mu_i}$ is $\mathcal{C}^1$ and
            positive definite on the domain $\Upsilon_i$
            containing a neighborhood of the origin,
   $0 \in \mathbb{R}^n$ is the global minimum where $\nabla V^{\mu_i}(0) = 0$.
   Also note that
   $g(x)$ and $\nabla V^{\mu_i}(x)$
   are continuous on the compact subset $\bar \Omega_{c_i}^{\mu_i}$
   ($\because$
        $\bar \Omega_{c_i}^{\mu_i} \subseteq \Upsilon_i \subseteq \mathcal{D}$,
        $g \in \mathcal{C}^0(\mathcal{D})$,
        and
        $V^{\mu_i} \in \mathcal{C}^1(\Upsilon_i)$).
   So, we have $\mu_{i+1}(0) = 0$
   from \eqref{eq:Policy Improvement} and
   $\nabla V^{\mu_i}(0) = 0$.
   From \eqref{eq:Policy Improvement}
        and
        the continuity of $\nabla V^{\mu_i}(x)$ and $g(x)$ on
        the compact subset $\bar \Omega_{c_i}^{\mu_i}$,
   it can be also shown that
   $\mu_{i+1}$ is continuous on $\bar \Omega_{c_i}^{\mu_i}$.
   Therefore, $\mu_{i+1}$ is a policy on $\bar \Omega_{c_i}^{\mu_i}$.

   For the proof of
    $\mu_{i+1} \in \mathcal{A}_\mathcal{I}( \bar \Omega_{c_i}^{\mu_i})$ and \eqref{eq:lemma:decreasing property of V},
   consider $V^{\mu_i}$ as a Lyapunov function candidate
   for the system $\dot x = f(x) + g(x)\mu_{i+1}(x)$.
   Differentiating $V^{\mu_i}(x)$
   with respect to the system $\dot x = f + g\mu_{i+1}$,
   we have
   \begin{align}
     \dot V^{\mu_i}(x)
     &= \nabla ^T V^{\mu_i}(x) \cdot \big ( f(x) + g(x)\mu_{i+1}(x) \big ) \notag \\
     &= - Q(x) - \mu_i^T R \mu_i - 2\mu_{i+1}^T R (\mu_{i+1} - \mu_i),
     \label{eq:lemma:admissible Omega i}
   \end{align}
   where \eqref{eq:Policy Evaluation} and \eqref{eq:Policy Improvement}
   are substituted in the second equality.
   Applying Young's inequality
   $2x^TRy \leq x^T R x + y^T R y$ for $x,y \in \mathbb{R}^m$
   to \eqref{eq:lemma:admissible Omega i},
   we obtain
   \begin{equation}
     \dot V^{\mu_i}(x) \leq - r(x, \mu_{i+1})<0,\;\; 
     \forall x \in  \Upsilon_i.
     \label{eq:lemma:dot V inequality}
   \end{equation}
   Therefore, by Theorem \ref{thm:Lyapunov thm}
   with $\Omega = \Upsilon_i$ and
        $\bar \Omega_{c} = \bar \Omega_{c_i}^{\mu_i} \subseteq \Upsilon_i$,
   $\mu_{i+1}$ asymptotically stabilizes the system $(f,g)$ on
   $\bar \Omega_{c_i}^{\mu_i}$,
   and $\bar \Omega_{c_i}^{\mu_i}$ is invariant under $\mu_{i+1}$, \emph{i.e.},
   \begin{equation}
     \textrm{if } x_0 \in  \bar \Omega_{c_i}^{\mu_i},
     \textrm{ then }
     \phi(t; x_0, \mu_{i+1}) \in  \bar \Omega_{c_i}^{\mu_i} \textrm{ for all } t \geq 0.
     \label{eq:lemma:invariance}
   \end{equation}
   Here,
   since we assume
    $\bar \Omega_{c_i}^{\mu_i} \subseteq \Upsilon_i$ and
    $V^{\mu_i} \in \mathcal{C}^1(\Upsilon_i)$,
   the invariance \eqref{eq:lemma:invariance} on $\bar \Omega_{c_i}^{\mu_i}$
   implies that
   $V^{\mu_i}(\phi(t; x_0, \mu_{i+1}))$ and $\dot V^{\mu_i}(\phi(t; x_0, \mu_{i+1}))$
   are finite for all $x_0 \in  \bar \Omega_{c_i}^{\mu_i}$ and all $t \geq 0$.
   So, one can integrate \eqref{eq:lemma:dot V inequality} from `$t=0$' to `$\infty$'
   to obtain
   \begin{align*}
     0  < V^{\mu_{i+1}}(x_0)
        &= \int_0^\infty r(\phi(\tau; x_0, \mu_{i+1}), \mu_{i+1}(\phi(\tau; x_0, \mu_{i+1}))) \; d\tau \\
        &\leq - \int_0^\infty \dot V^{\mu_i}(\phi(\tau; x_0, \mu_{i+1})) \; d\tau
        = V^{\mu_i}(x_0) < \infty,
   \end{align*}
   where we have used
   $
    \lim_{t \to \infty} V^{\mu_i}(\phi(t; x_0, \mu_{i+1})) = 0
   $
   in the equality,
   which holds $\forall x_0 \in \bar \Omega_{c_i}^{\mu_i}$
   since $V^{\mu_i}(0) = 0$ and
   \[
    \lim_{t \to \infty} \phi(t; x_0, \mu_{i+1}) = 0 \;\; \forall x_0 \in \bar \Omega_{c_i}^{\mu_i}
   \]
   by asymptotic stability.
   Therefore,
   $V^{\mu_{i+1}}$ satisfies $0 < V^{\mu_{i+1}}(x_0) \leq V^{\mu_{i}}(x_0) < \infty$
   $\forall x_0 \in  \bar \Omega_{c_i}^{\mu_i}$.
   This implies
   $\mu_{i+1} \in \mathcal{A}(\bar \Omega_{c_i}^{\mu_i})$,
   and we have $\mu_{i+1} \in \mathcal{A}_\mathcal{I}( \bar \Omega_{c_i}^{\mu_i})$
   by \eqref{eq:lemma:invariance}.
  \end{proof}


 If a feasible subset $\Upsilon_i$ of $R_A(\mu_i)$ is given \emph{a priori},
 it can be used to determine the invariant admissible region
 $\bar \Omega_{c_i}^{\mu_i} \subseteq \Upsilon_i$.
 Moreover, if the domain $\mathcal{D}$ is extended to satisfy
 $R_A(\mu) \subseteq \mathcal{D}$ at $i$-th iteration,
 then $\Upsilon_i$ can be given as the largest attractive set
 $\Upsilon_i = R_A(\mu_i)$ for the policy $\mu_i$.
 In this case, $\Upsilon_i$ is also a feasible subset of $\mathcal{D}$
 since $R_A(\mu)$ is itself an invariant set \cite{Khalil2002} and contained by $\mathcal{D}$.
 However, calculating the region of attraction $R_A(\mu_i)$ at each $i$-th iteration
 or its feasible subset
 is not a trivial task and needs high computational burden.
 To avoid such difficulties,
 the admissible set $\Omega_i \in \mathcal{D}$ given \emph{a priori}
 can be used as a feasible subset $\Upsilon_i$ of $R_A(\mu_i)$,
 \emph{i.e.}, $\Upsilon_i = \Omega_i$.
 The following corollary shows that
 under the assumption
 \begin{assumption}
 \label{assumption:C1 of V_i}
    For each $\mu_{i} \in \mathcal{A}(\Omega_i)$,
    $V^{\mu_i}$ is continuously differentiable on $\Omega_i$, i.e., $V^{\mu_i} \in \mathcal{C}^1(\Omega_i)$;
 \end{assumption}
 \noindent
 this choice ``$\Upsilon_i = \Omega_i$'' is reasonable.
 The proof of the corollary can be easily done by applying Theorem \ref{thm:main thm1} with $\Upsilon_i = \Omega_i$
 under Assumption \ref{assumption:C1 of V_i}.

 \begin{corollary}
   \label{cor:admissibility of the next policy}
    Assume $\mu_i \in \mathcal{A}(\Omega_i)$
    and $c_i$ is chosen such that $\bar \Omega_{c_i}^{\mu_i} \subseteq \Omega_i$.
    Then,
    \begin{enumerate} \itemsep-0.3em
    \item
    $\mu_{i+1}$ is a policy on $\bar \Omega_{c_i}^{\mu_i}$;
    \item
    $\mu_i, \mu_{i+1} \in \mathcal{A}_\mathcal{I}( \bar \Omega_{c_i}^{\mu_i})$;
    \item
    $
        0 < V^{\mu_{i+1}}(x) \leq V^{\mu_{i}}(x) < \infty
    $
    for all $x \in  \bar \Omega_{c_i}^{\mu_i}$.
    \end{enumerate}
 \end{corollary}

  Furthermore, the next theorem
  states that the (invariant) admissibility and
  the value function decreasing property are preserved
  under $\Omega_{i+1}$ determined by $\Omega_{i+1} = \bar \Omega_{c_i}^{\mu_i}$
  and $c_i > 0$ chosen such that
  $\bar \Omega_{c_i}^{\mu_i} \subseteq \Omega_i$ is guaranteed.

  \begin{theorem}
  \label{thm:main thm1-1}
  Assume the initial policy $\mu_0$ is admissible
  on $\Omega_0 \subseteq \mathcal{D}$.
  If the policies $\{\mu_i\}$ and
  the value functions $\{V^{\mu_i}\in \mathcal{C}^1\}$
  are generated by the proposed PI (Fig. \ref{fig:modified PI})
  with $(\Omega_{i+1}, c_i)$
  determined at each $i$-th step
  such that
    $\Omega_{i+1} = \bar \Omega_{c_i}^{\mu_i}$ and
    $\bar \Omega_{c_i}^{\mu_i} \subseteq \Omega_i$,
  then for all $i \in \{0,1,2, \cdots\}$,
  \begin{enumerate} \itemsep-0.3em
  \item
  $\mu_{i+1}$ is a policy on $\Omega_{i+1}$
  \item
  $\mu_i$ and $\mu_{i+1}$ are invariantly admissible on $\Omega_{i+1}$;
  \item
  for all $x \in \Omega_{i+1}$,
  \begin{equation}
  \label{eq:decreasing V}
  0 < V^{\mu_{i+1}}(x) \leq V^{\mu_{i}}(x) \leq \cdots \leq V^{\mu_{0}}(x).
  \end{equation}
  \end{enumerate}
  \end{theorem}

  \begin{proof}
  Since we assume $\mu_0 \in \mathcal{A}(\Omega_0)$
  and $\Omega_1 = \bar \Omega_{c_0}^{\mu_0} \subseteq \Omega_0$,
  Corollary \ref{cor:admissibility of the next policy} implies
  \begin{enumerate} \itemsep-0.3em
    \item
    $\mu_{1}$ is a policy on $\Omega_1$;
    \item
    $\mu_0, \mu_{1} \in \mathcal{A}_\mathcal{I}(\Omega_1)$;
    \item
    $V^{\mu_{1}}$ satisfies
    $0 < V^{\mu_{1}}(x) \leq V^{\mu_{0}}(x) < \infty$
    for all $x \in  \Omega_{1}$.
  \end{enumerate}
  Then, we have again $\mu_{1} \in \mathcal{A}(\Omega_1)$
  by Proposition \ref{prop:A to AI}.
  Repeating this process $i$-times, we can prove the first and second parts;
  this process also proves that for any $j \in \{0,1,2,\cdots,i-1,i\}$,
  \begin{equation}
    0 < V^{\mu_{j+1}}(x) \leq V^{\mu_{j}}(x) \;\;\;
    \label{eq:main thm1-1:contraction of V}
  \end{equation}
  is satisfied for all $x \in \Omega_{j+1}$.
  Finally, since $\Omega_{i+1}$ satisfies
  \[
        \Omega_{i+1} \subseteq \cdots \subseteq \Omega_{j+1} \subseteq \Omega_j \subseteq \cdots \subseteq \Omega_1 \subseteq \Omega_0,
  \]
  \eqref{eq:main thm1-1:contraction of V} also holds
  for all $x \in \Omega_{i+1}$ ($\because \Omega_{i+1} \subseteq \Omega_{j+1}$)
  and all $j \in \{0,1,2,\cdots, i\}$,
  which completes the proof of the third statement.
  \end{proof}

 \subsection{Convergence analysis}
 Now, we analyze the convergence properties of
 the proposed PI method (Fig. \ref{fig:modified PI}) under the assumption that
 \begin{assumption}
 \label{assumption:Omega and c}
    The initial admissible region $\Omega_0$
    satisfies $\Omega_0 \subseteq \Omega^*$,
    and $(\Omega_{i+1}, c_i)$ is determined for all $i \in \{0,1,2,\cdots\}$ such that $\Omega_{i+1} = \bar \Omega_{c_i}^{\mu_i}$
          and
         $\bar \Omega_{c_i}^{\mu_i} \subseteq \Omega_i$.
    That is,
    \[
        \Omega_{i+1} \subseteq \Omega_i \subseteq \cdots \subseteq \Omega_0 \subseteq \Omega^* \subseteq \mathcal{D}, \;\;
        \forall i \in \{0,1,2,\cdots\}.
    \]
 \end{assumption}

 \begin{theorem}
 \label{thm:conv1}
  Consider
  policies $\{\mu_i\}$ and
  value functions $\{V^{\mu_i}\in \mathcal{C}^1\}$
  generated by the invariantly admissible PI
  under an admissible initial policy $\mu_0$
  and
  Assumptions \ref{assumption:Lipschitz}--\ref{assumption:Omega and c}.
  If for some $k \in \{0,1,2\cdots\}$,
  $V^{\mu_k}(x) = V^{\mu_{k+1}}(x)$ holds for all $x \in \Omega_{k+1}$,
  then
  \[
    \mu_k = \mu_{k+1} = \mu^* \textrm{ and } V^{\mu_k} = V^{\mu_{k+1}}= V^* \textrm{ on } \Omega_{k+1}.
  \]
 \end{theorem}

 \begin{proof}
    Substituting $V^{\mu_k} = V^{\mu_{k+1}}$ and \eqref{eq:Policy Improvement}
    for $i=k$
    into \eqref{eq:Policy Evaluation} for $i=k+1$, we have
    \begin{align}
    Q(x) + (\nabla {V^{\mu_{k+1}}}(x))^T f(x)
    &-\frac{1}{4} (\nabla {V^{\mu_{k+1}}}(x))^T g(x) R^{-1} g^T(x) \nabla V^{\mu_{k+1}}(x) = 0,
    \label{eq:convergence1:HJB}
    \end{align}
    that holds for all $x \in \Omega_{k+1}$
    by Assumptions \ref{assumption:C1 of V_i} and \ref{assumption:Omega and c}.
    Note that \eqref{eq:convergence1:HJB} is the HJB equation \eqref{eq:HJB}.
    Since $\Omega_{k+1} \subseteq \Omega^*$
    by Assumption \ref{assumption:Omega and c}
    and $V^*$ is the unique $\mathcal{C}^1$-solution of the HJB equation \eqref{eq:HJB}
    over $\Omega^*$
    by Assumption \ref{assumption:HJB},
    we have $V^{\mu_k} = V^{\mu_{k+1}} = V^*$ on $\Omega_{k+1}$.
    Moreover, \eqref{eq:optimal policy u*}  and
    \eqref{eq:Policy Improvement} with $i = k$ and $\mu_{k+1} = \mu_k$
    proves $\mu_k = \mu_{k+1} = \mu^*$ on $\Omega_{k+1}$.
 \end{proof}

 Theorem \ref{thm:conv1} states that
 if the process of invariantly admissible PI is terminated
 by convergence in a finite number of steps, then
 the solution is guaranteed to be optimal.
 To investigate the general convergence conditions
 in case of that the process does not end in a finite number of steps,
 define the limit set $\hat \Omega$ as
 $\hat \Omega := \bigcap_{i=0}^\infty \Omega_i$.
 Then,
 from \eqref{eq:decreasing V} and $\Omega_{i+1} = \bar \Omega_{c_i}^{\mu_i}$,
 one can see that, for the condition
  \begin{equation}
     \hat \Omega \subseteq \cdots \subseteq \Omega_{i+1} \subseteq \Omega_{i} \subseteq \cdots \subseteq \Omega_1 \subseteq \Omega_0
     \label{eq:monotonicity of Omega sets}
  \end{equation}
 in Assumption \ref{assumption:Omega and c},
 $\{c_i > 0 \}$ should be monotonically decreasing.
 In this case, since $\{c_i\}$ is bounded by zero,
 it converges with this decreasing condition
 to a limit point $\hat c := \lim_{i \to \infty} c_i$
 in a decreasing order
 \begin{equation}
   0 < c_{i+1} \leq c_i \leq \cdots \leq c_1 \leq c_0.
   \label{eq:ci monotone decreasing}
 \end{equation}
 Also note that the limit set $\hat \Omega$ is compact
 since arbitrary intersection of the closed and bounded sets $\Omega_i$
 is also closed and bounded.


 \begin{lemma}
 \label{lemma:conv}
  Under $\mu_0 \in \mathcal{A}(\Omega_0)$ and
  Assumptions \ref{assumption:Lipschitz}--\ref{assumption:Omega and c},
  there is a function $\hat V: \hat \Omega \to \mathbb{R}_+$
  such that
  $\{V^{\mu_i}\in \mathcal{C}^1\}$ generated by the invariantly admissible PI
  pointwisely converges to $\hat V$
  on the limit set $\hat \Omega$ as $i \to \infty$,
  in a decreasing order
  \begin{equation}
        0 \leq \hat V(x) \leq V^{\mu_{i+1}}(x) \leq V^{\mu_{i}}(x).
        \label{eq:lemma:conv:decreasing}
  \end{equation}
 \end{lemma}

 \begin{proof}
   By \eqref{eq:decreasing V} in Theorem \ref{thm:main thm1-1}
   and $\hat \Omega \subseteq \Omega_{i+1}$,
   we have $0 < V^{\mu_{i+1}}(x) \leq V^{\mu_{i}}(x)$
   for all $x \in \hat \Omega$ and all $i \in \{0,1,2,\cdots\}$.
   Therefore, for any fixed $x \in \hat \Omega$,
   $\{V^{\mu_{i}}(x)\}$ is decreasing and bounded by zero,
   implying the existence of $\hat V$
   to which $\{V^{\mu_{i}}\}$ pointwisely converges
   in the decreasing order \eqref{eq:lemma:conv:decreasing},
   which completes the proof.
 \end{proof}

  The next theorem states
  the conditions for
  the uniform monotonic convergence of
  $\{V^{\mu_i}\}$ and $\mu_i$
  to the optimal solution $V^*$ and $\mu^*$
  on the limit set $\hat \Omega$, respectively.
  For the discussion,
  we denote by $\mathcal{C}^1_\mathcal{A} (\hat \Omega)$ 
  the set of all continuously differentiable value functions
  for the policies that are admissible on the compact limit set
  $\hat \Omega$. That is,
  \begin{equation*}
    \mathcal{C}^1_\mathcal{A} (\hat \Omega)
    :=
    \{ V^\mu \in \mathcal{C}^1(\Omega): \mu \in \mathcal{A}(\hat \Omega) \textrm{ and } \phi(t;x_0, \mu) \in \Omega, \; \forall x_0 \in \hat \Omega, \; \forall t \geq 0 \}.
  \end{equation*}
  Note that $V^{\mu_i} \in \mathcal{C}^1(\Omega_i)$
  generated by the proposed algorithm belongs to $\mathcal{C}^1_\mathcal{A}(\hat \Omega)$
  for any $i \in \{0,1,2,\cdots\}$.
  This is because
  ``$\hat \Omega \subseteq \Omega_i$
  (see \eqref{eq:monotonicity of Omega sets} or Assumption \ref{assumption:Omega and c})
  and $\mu_i \in \mathcal{A}(\Omega_i)$''
  implies $\mu_i \in \mathcal{A}(\hat \Omega)$,
  so
  $V^{\mu_i}$ is a value function
  for $\mu_i \in \mathcal{A}(\hat \Omega) \subseteq \mathcal{A}(\Omega_i)$.

  Next, we define PI operator
  $\mathcal{T}: \mathcal{C}^1_\mathcal{A} (\hat \Omega) \to \mathcal{C}^1_\mathcal{A} (\hat \Omega)$
  as a composite mapping
  $\mathcal{T} = \mathcal{T}^E \circ \mathcal{T}^{I}$,
  where $\mathcal{T}^{E}$ and $\mathcal{T}^{I}$
  are policy evaluation and improvement operators
  defined as follows.
  \begin{enumerate}
    \item
      $\mathcal{T}^{E}:
      \mathcal{A}(\hat \Omega) \rightarrow \mathcal{C}^1_{\mathcal{A}}(\hat \Omega)$
      is a map from an admissible policy $\mu \in \mathcal{A}(\hat \Omega)$
      to the corresponding value function
      $V^\mu \in \mathcal{C}^1_{\mathcal{A}}(\hat \Omega)$
      satisfying \eqref{eq:Lyapunov equation for mu}.
      That is,
      \[
        V^{\mu} = \mathcal{T}^{E}(\mu).
      \]
    \item
      $\mathcal{T}^{I}:
      \mathcal{C}^1_{\mathcal{A}}(\hat \Omega) \rightarrow \mathcal{A}(\hat \Omega)$
      is a map from
      a value function $V^\mu \in \mathcal{C}^1_{\mathcal{A}}(\hat \Omega)$ to
      the admissible policy
      $\mu^+ \in \mathcal{A}(\hat \Omega)$
      satisfying $\mu^+ = -{1 \over 2} R^{-1} g^T \nabla V^\mu$.
      That is,
      \[
        \mu^+ = \mathcal{T}^{I}(V^{\mu}).
      \]
  \end{enumerate}

  So, for a given value function $V^\mu \in \mathcal{C}^1_{\mathcal{A}}(\hat \Omega)$
  \textrm{ satisfying } $(\nabla V^\mu)^T (f + g\mu) = 0$, $\mathcal{T}$
  yields the value function
  $V^{\mu^+} \in \mathcal{C}^1_{\mathcal{A}}(\hat \Omega)$
  satisfying $(\nabla V^{\mu^+})^T (f + g\mu^+) = 0$
  for the improved admissible policy
    $\mu^+ = -{1 \over 2} R^{-1} g^T \nabla V^\mu$.
  In a compact form,
  \[
        V^{\mu^+} = \mathcal{T}(V^\mu).
  \]
  Note that $\mathcal{T}$ represents one cycle of policy evaluation and improvement;
  the value functions $V^{\mu_i}$ and $V^{\mu_{i+1}}$
  generated by the PI method satisfy $V^{\mu_{i+1}} = \mathcal{T}(V^{\mu_i})$.
  Also note that if $V^{\mu_k} = V^{\mu_{k+1}}$, then
  $\mathcal{T}$ satisfies $\mathcal{T}(V^{\mu_k}) =  V^{\mu_{k}}$.
  This implies that the fixed point of the operator $\mathcal{T}$
  corresponds to the optimal value function $V^*$
  since $V^{\mu_k} = V^{\mu_{k+1}}$ implies $V^{\mu_k} = V^{\mu_{k+1}} = V^*$
  by Theorem \ref{thm:conv1}.

  \begin{theorem}
  \label{thm:main thm2}
  Suppose $\mathcal{T}$ is continuous
        and
        the limit function $\hat V$ in \eqref{eq:lemma:conv:decreasing}
        belongs to
        $\mathcal{C}^1_A(\hat \Omega)$.
  Then,
  the sequence of value functions $\{V^{\mu_i}\in \mathcal{C}^1\}$
  generated by the invariantly admissible PI
  under $\mu_0 \in \mathcal{A}(\Omega_0)$ and
  Assumptions \ref{assumption:Lipschitz}--\ref{assumption:Omega and c}
  uniformly converges to the optimal solution $V^*$
  on $\hat \Omega$ in a decreasing order
  \begin{equation}
          0 < V^{*}(x) \leq \cdots \leq V^{\mu_{i+1}}(x) \leq V^{\mu_{i}}(x) \leq \cdots \leq V^{\mu_{0}}(x).
          \label{eq:main thm2:conv:decreasing}
  \end{equation}
  \end{theorem}

  \begin{proof}
    First, note that
    Lemma \ref{lemma:conv} guarantees
    the existence of the limit function $\hat V$ to which
    $\{V^{\mu_i} \in \mathcal{C}^1 \} \subseteq \mathcal{C}^1_A(\hat \Omega)$
    converges pointwisely in a decreasing order \eqref{eq:lemma:conv:decreasing}.
    Since we assume $\hat V \in \mathcal{C}^1_A(\hat \Omega)$,
    $\hat V$ is continuous on $\hat \Omega$.
    So, the convergence $V^{\mu_i} \to \hat V$ is uniform
    on the compact set $\hat \Omega$ by Dini's theorem.
    Similarly, we have the uniform convergence $\mathcal{T}(V^{\mu_i}) = V^{\mu_{i+1}} \to \hat V$ on $\hat \Omega$.
    Therefore, $\mathcal{T}(\hat V) = \hat V$ by continuity of $\mathcal{T}$,
    \emph{i.e.}, $\hat V$ is the fixed point of $\mathcal{T}$.
    Since the fixed point of $\mathcal{T}$ corresponds to the optimal solution $V^*$
    and $\hat \Omega \subseteq \Omega^*$ by Assumption \ref{assumption:Omega and c},
    we have $\hat V = V^*$ on $\hat \Omega$, and \eqref{eq:main thm2:conv:decreasing}
    can be obtained from \eqref{eq:lemma:conv:decreasing}.
  \end{proof}

  \begin{remark}
  In the convergence analysis \cite{Beard1995} of the SAM given in \cite{Beard1995,Beard1997},
  it was implicitly assumed that $\hat V$ is continuously differentiable
  and that $\hat V$ satisfies $H(x, \hat \mu, \nabla \hat V) = 0$
  where $\hat \mu$ is given by $\hat \mu = -R^{-1}g^T \nabla \hat V$
  (see Theorem 5.3.1 in \cite{Beard1995} and its proof).
  The same assumptions exist in the convergence proofs of the variants \cite{Abu-Khalaf2005,Saridis1979,Murray2003}.
  However, the convergence of $C^1$ functions $V^{\mu_i} \to \hat V$
  proven in Lemma \ref{lemma:conv}
  and \cite{Abu-Khalaf2005,Beard1995,Saridis1979,Murray2003}
  does not imply the convergence $\nabla V^{\mu_i} \to \nabla \hat V$ in general,
  and
  cannot guarantee even the differentiability of the limit function $\hat V$ (see \cite{Marsden1993}).
  In Theorem \ref{thm:main thm2} of this paper,
  we have exactly and rigorously stated
  the conditions for $\hat V = V^*$ (and uniform convergence $V^{\mu_i} \to V^*$)
  regarding the proposed invariantly admissible PI;
  similar conditions were given only in \cite{Leake1967} for the usual SAM
  without admissible region update.
  \end{remark}

  \subsection{Determination of $c_i > 0$ of $\bar \Omega_{c_i}^{\mu_i}$ ($=\Omega_{i+1}$)}

  Under Assumption \ref{assumption:Omega and c},
  the region $\Omega_{i}$ becomes more conservative
  as the learning continues (see also \eqref{eq:ci monotone decreasing}).
  That is, as can be seen from Fig. \ref{fig:SetOfOmegas},
  $\bar \Omega_{c_{i}}^{\mu_{i}}$ ($=\Omega_{i+1}$)
  is necessarily smaller than or equal to
  both $\bar \Omega_{c_{i-1}}^{\mu_{i}}$ and $\bar \Omega_{c_{i-1}}^{\mu_{i-1}}$ ($=\Omega_i$).
  Here, the set $\bar \Omega_{c_{i-1}}^{\mu_{i}}$ is obviously larger than or equal to
  $\bar \Omega_{c_{i-1}}^{\mu_{i-1}}$ for the same $c_{i-1}$ by the monotonic decreasing property \eqref{eq:decreasing V}.

  We now propose
  an invariant admissible region update rule
  to alleviate the conservativeness of the next region $\Omega_{i+1}$.
  Under Assumption \ref{assumption:Omega and c},
  the proposed update rule determines at each $i$-th iteration
  the largest region $\bar \Omega_{c_i^*}^{\mu_i}$
  contained by $\Omega_i$ ($=\bar \Omega_{c_{i-1}}^{\mu_{i-1}}$ for $i \geq 1$).
  The update starts with an initial admissible region given by $\Omega_0 = \bar B_0(r)$,
  where $\bar B_0(r)$ is a closed-ball at the origin with $r$
  determined to satisfy $\mu_0 \in \mathcal{A}(\bar B_0(r))$.
  Then, at each $i$-th iteration,
  the update rule determines the radius $c_{i}^*$ of the next region $\Omega_{i+1}$ by
  \begin{equation}
    c_{i}^* = \min \big \{ V^{\mu_{i}}(x): x \in \partial \Omega_i \big \},
    \label{eq:c star equation}
  \end{equation}
  where
  $\partial \Omega_i$ is the boundary of $\Omega_i$.
  With this $c_i^*$, the next region is updated by $\Omega_{i+1} = \bar \Omega_{c_i^*}^{\mu_i}$

  The maximum region $\bar \Omega_{c_{i}^*}^{\mu_{i}}$ and the normal region $\bar \Omega_{c_{i}}^{\mu_{i}}$
  at $i$-th step are shown in Fig. \ref{fig:SetOfOmegas}.
  Compared to the normal one,
  $\bar \Omega_{c_{i}^*}^{\mu_{i}}$ has the maximum radius $c_{i}^* > 0$
  on the constrained set $\Omega_i$
  ($=\bar \Omega_{c_{i-1}}^{\mu_{i-1}}$ for $i \geq 1$)
  while satisfying $\bar \Omega_{c_{i}}^{\mu_{i}} \subseteq \bar \Omega_{c_{i}^*}^{\mu_{i}}$
  without violating $\Omega_{i+1} \subseteq \Omega_i$ in Assumption \ref{assumption:Omega and c}
  as illustrated in Fig. \ref{fig:SetOfOmegas}.

  \begin{remark}
    \label{remark:4}
    Even if the region $\Omega_{i+1}$ is updated by $\Omega_{i+1}= \bar \Omega_{c_{i}^*}^{\mu_{i}}$ with \eqref{eq:c star equation},
    it may become very small or narrow in some cases by \eqref{eq:monotonicity of Omega sets} as $i$ increases.
    In this case, $\Omega_{i+1}$ can be enlarged at some $i$-th update step
    by calculating a larger feasible compact subset $\Upsilon_i$ of $R_A(\mu_i)$ such that $\Omega_i \subseteq \Upsilon_i \subseteq \mathcal{D}$
    and then determining $\Omega_{i+1}$ by $\Omega_{i+1} = \bar \Omega_{\alpha_{i}^*}^{\mu_{i}} $ with $\alpha_i^*$ chosen by
    \begin{equation}
        \alpha_{i}^* = \min \big \{ V^{\mu_{i}}(x): x \in \partial \Upsilon_i \big \}.
        \label{eq:alpha star equation}
    \end{equation}
    In this case, $\mu_{i+1} \in \mathcal{A}_\mathcal{I}(\Omega_{i+1})$ is guaranteed by Theorem \ref{thm:main thm1},
    and by $\Omega_i \subseteq \Upsilon_i$, we have $c_i^* \leq \alpha_i^*$.
    Therefore,
    the next domain $\Omega_{i+1}$ updated by \eqref{eq:alpha star equation} is essentially larger than
    that updated by \eqref{eq:c star equation}, resulting in the larger final domain $\hat \Omega$ at last.
  \end{remark}
  \begin{figure} [t]
    \centering
    \epsfig{figure=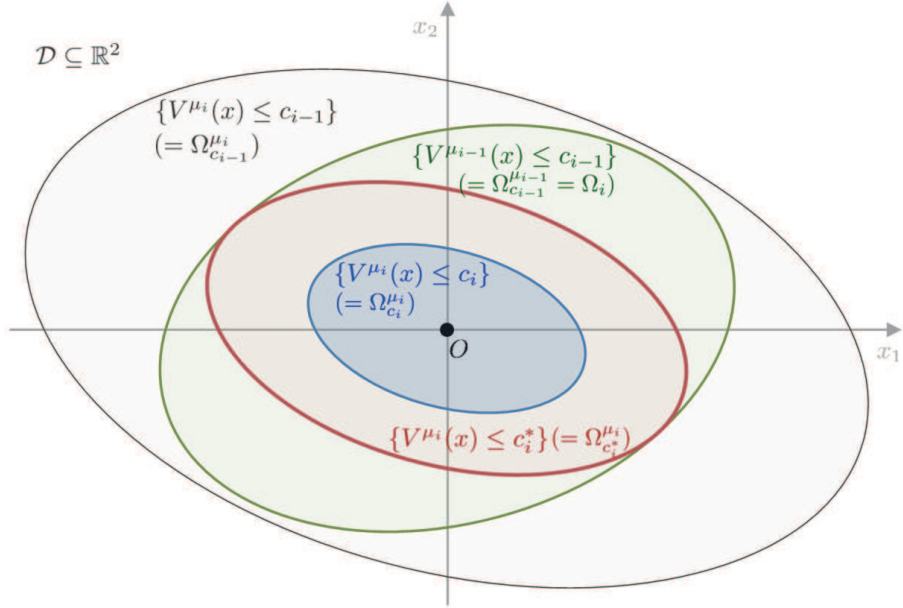, width=12cm}
    \caption{An illustration of $\bar \Omega^{\mu}_c$-sets in $\mathbb{R}^2$ and their relations.}
    \label{fig:SetOfOmegas}
  \end{figure}

 \begin{remark}
   As mentioned in Section 1,
   there are many PI algorithms \cite{Lee2012_WCCI,Lewis2009,Murray2002,Vrabie2009b,Lee2012}
   for the optimal control problem \eqref{eq:system}--\eqref{eq:performance measure},
   ideally generating the same policy and value function sequences
   ($\{\mu_i\}_{i=0}^\infty$ and $\{V^{\mu_i}\}_{i=0}^\infty$)
   as the SAM \cite{Beard1995,Beard1997}.
   So, the proposed invariantly admissible PI method can be easily extended
   to those equivalent PI and reinforcement learning algorithms
   to improve the closed-loop stability,
   by incorporating the (invariant) admissible region update step and the update rules
   \eqref{eq:c star equation} and \eqref{eq:alpha star equation} into those algorithms.
 \end{remark}


  \begin{figure} [h]
  \centering
  \subfigure[Evolution of $\Omega_i$]{\epsfig{figure=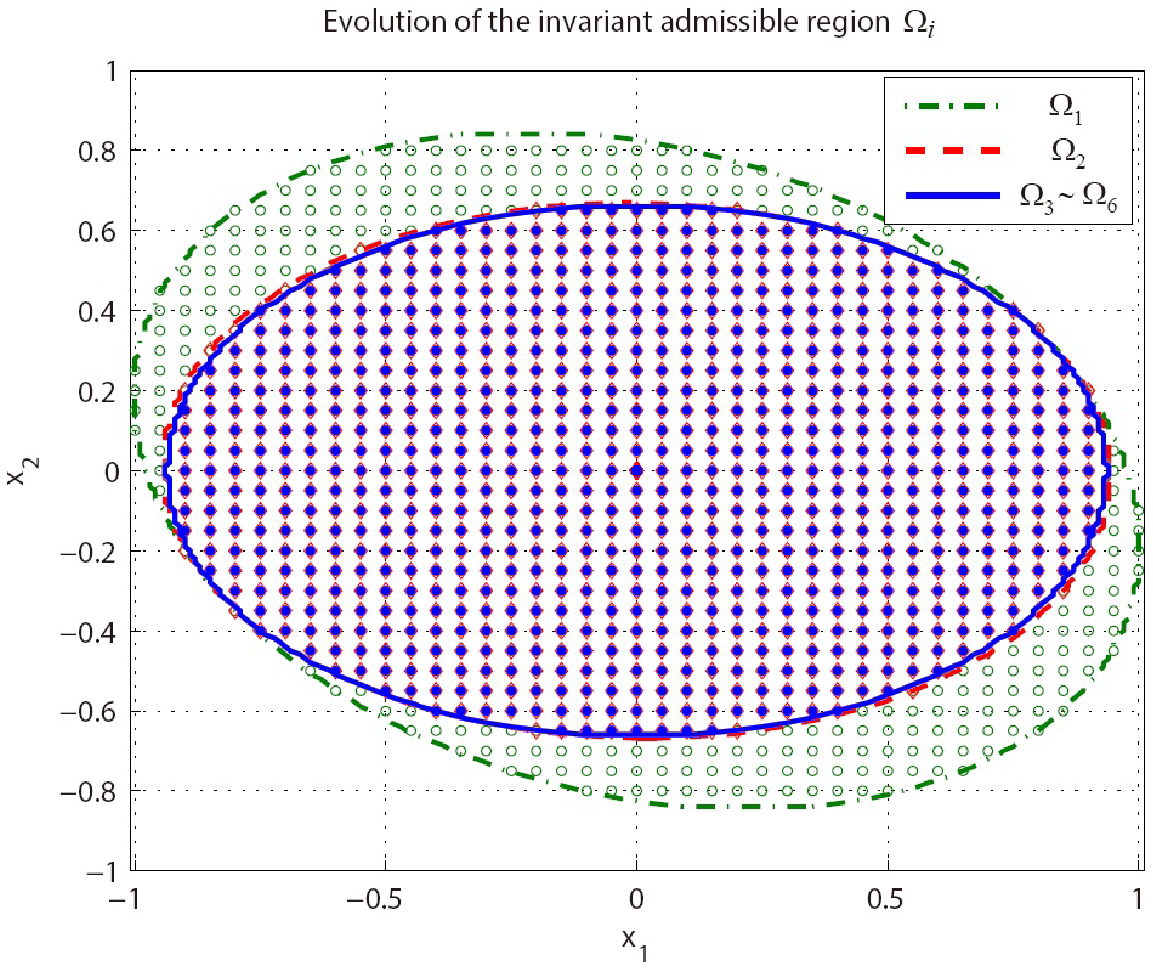,width=10cm}}
  \subfigure[Evolution of $w_j$ ($j=1,2,3$)]{\epsfig{figure=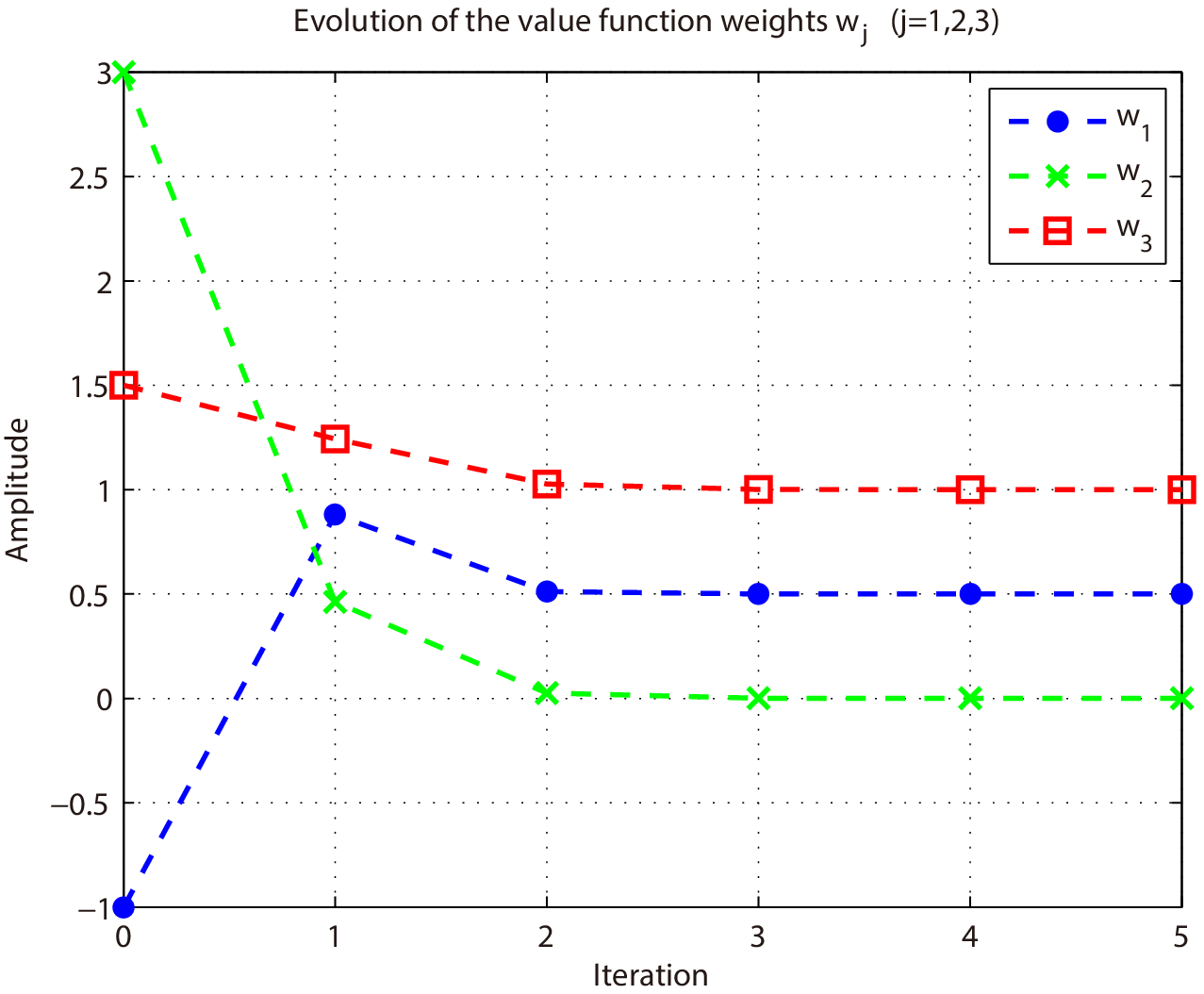,width=10cm}}
  \caption{Simulation results for nonzero case $w_{1,0} = -1$, $w_{2,0} = 3$, and $w_{3,0} = 1.5$:
  (a) evolution of $\Omega_i$,
  (b) evolution of the weights $w_j$ ($j = 1,2,3$).}
  \label{fig:sim1}
  \end{figure}

  \begin{figure} [h]
  \centering
  \subfigure[Evolution of $\Omega_i$]{\epsfig{figure=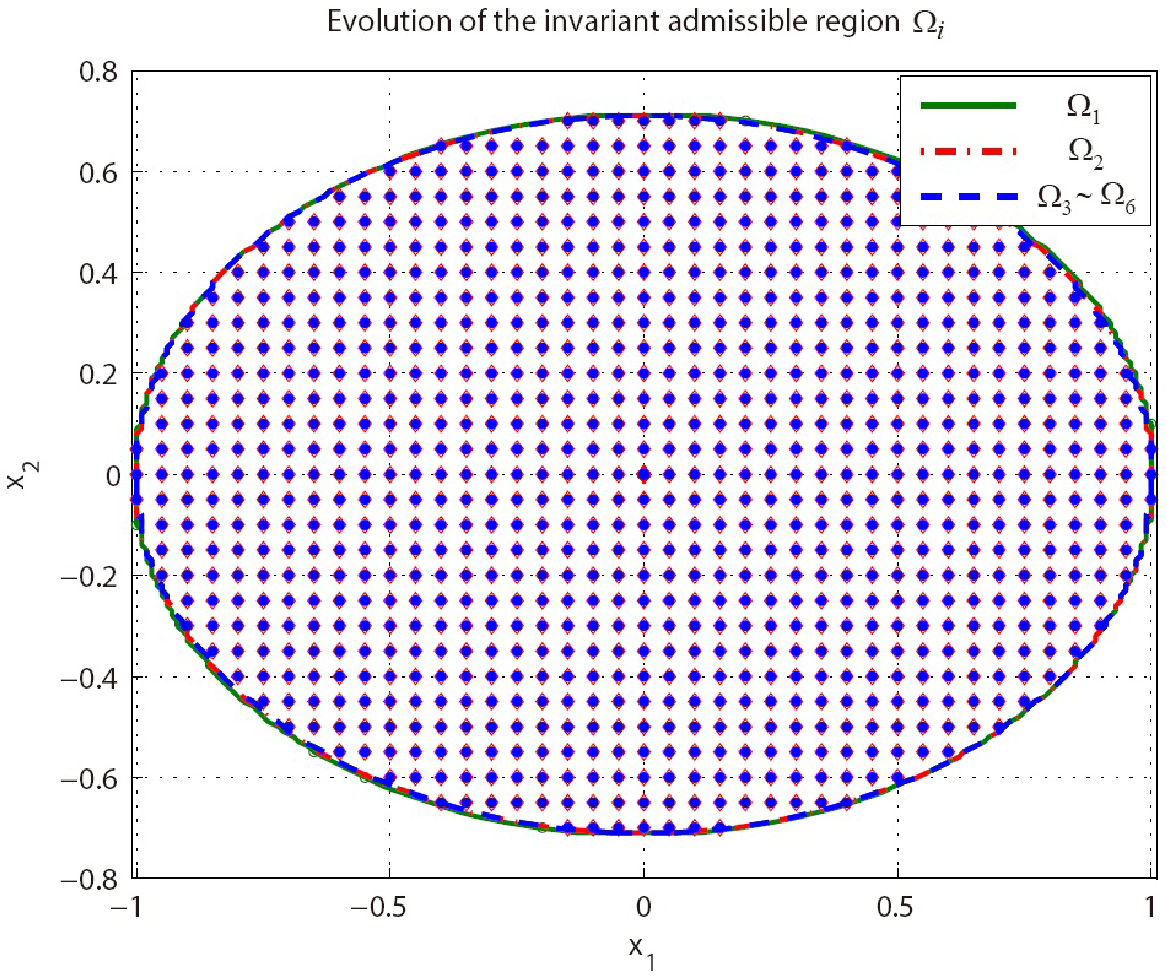,width=10cm}}
  \subfigure[Evolution of $w_j$ ($j=1,2,3$)]{\epsfig{figure=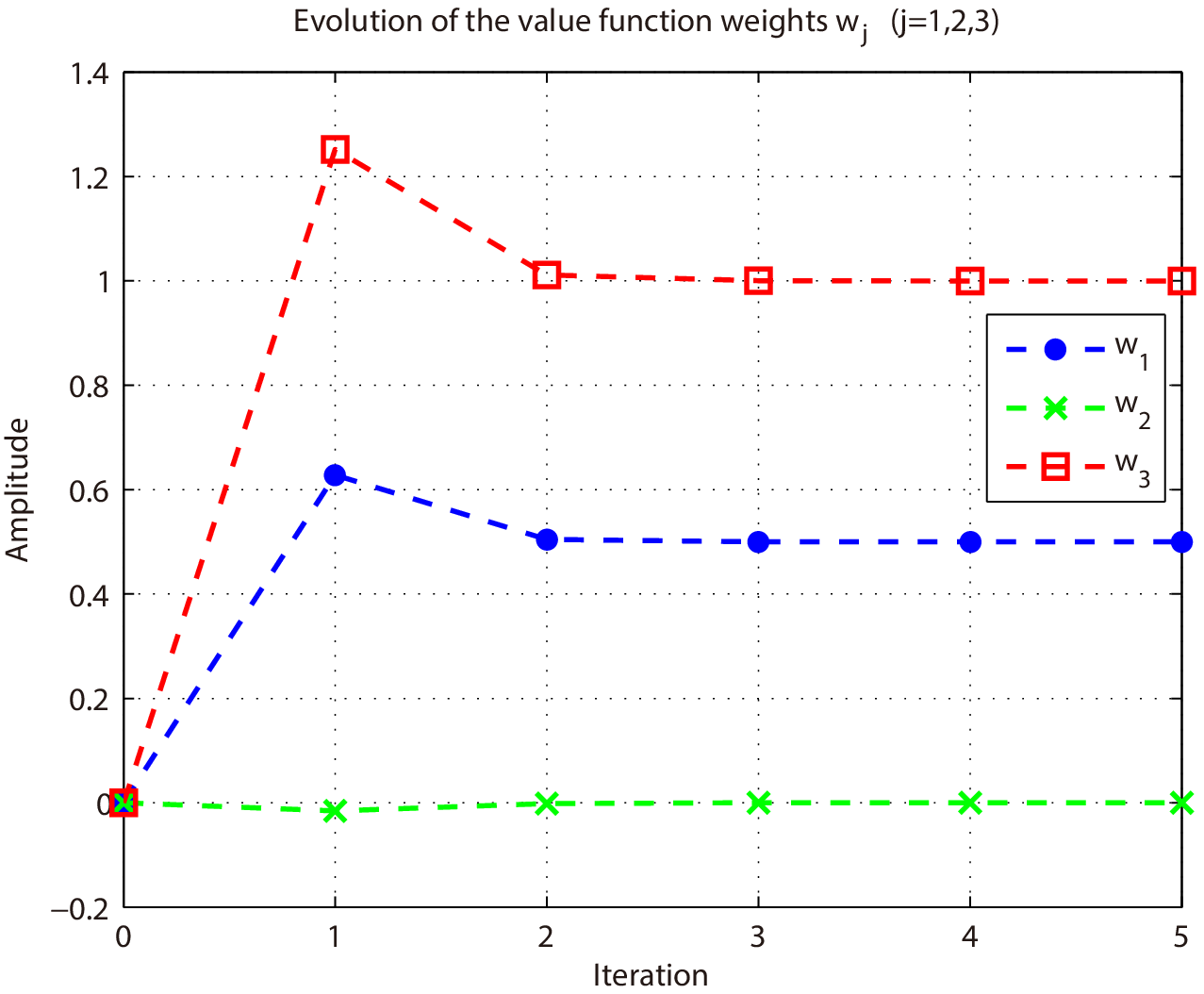,width=10cm}}
  \caption{Simulation results with zero initial weights $w_{1,0} = -1$, $w_{2,0} = 3$, and $w_{3,0} = 1.5$:
  (a) evolution of $\Omega_i$,
  (b) evolution of the weights $w_j$ ($j = 1,2,3$).}
  \label{fig:sim2}
  \end{figure}

\section{Numerical simulations}
  To illustrate the proposed PI method and its effectiveness,
  we performed the numerical simulations for the following nonlinear system:
  \begin{equation}
      \begin{cases}
      \dot x_1 = -x_1 + x_2\\
      \dot x_2 = -(x_1+x_2)/2 + x_2 (\sin^2 x_1)/2 + (\sin x_1)u
      \end{cases}
  \label{eq:sim:target nonlinear system}
  \end{equation}
  and the performance measure \eqref{eq:performance measure}
  with $r(x,u) = x_1^2 + x_2^2 + u^2$.
  This optimal control problem was also shown
  in \cite{Vrabie2009b} to simulate their nonlinear integral PI method
  which, in ideal cases, generates the same sequences of policies and value functions
  to the SAM in \cite{Beard1995,Beard1997}.
  Using the converse HJB approach \cite{Nevistic1996},
  the optimal solution $(V^*, \mu^*)$ is given by
  \[
      \mu^*(x)=- x_2 \sin x_1 \;\; \textrm{ and } \;\; V^*(x) = \frac{1}{2}x_1^2 + x_2^2.
  \]
  As in \cite{Vrabie2009b}, the value function $V^{\mu}(x)$ is
  parameterized as
  $V^{\mu}(x) = w_1 x_1^2 + w_2 x_1 x_2 + w_3 x_2^2$,
  where $w_j$ ($j = 1,2,3$) are the weights to be determined in policy evaluation
  of the proposed PI method at every iteration.

  In the simulations,
  the sample points $x$ used in the $i$-th policy evaluation step
  are collected only in the $i$-th admissible set $\Omega_i$ of the state space,
  with the same sampling interval $\Delta x_1 = \Delta x_2 = 10^{-2}$.
  Here, the initial admissible region is given by $\Omega_0 = \bar B_{0}(1)$
  as in \cite{Vrabie2009b},
  where $\bar B_{0}(1) = \{x \in \mathbb{R}^2: |x_1| \leq 1, |x_2| \leq 1\}$.
  Then, the invariant admissible region update agent in the simulations
  finds, at each $i$-th step,
  the optimal radius $c^*_i$ satisfying \eqref{eq:c star equation}
  to determine the next largest invariant admissible region $\Omega_{i+1} = \bar \Omega_{c^*_i}^{\mu_i} \subseteq \Omega_i$
  on the constraint set $\Omega_i$.

  Fig. \ref{fig:sim1} illustrates the simulation result
  for nonzero initial weights $w_{1,0} = -1$, $w_{2,0} = 3$, and $w_{3,0} = 1.5$.
  As can be seen from Fig. \ref{fig:sim1}(b), the weights $w_j$ ($j=1,2,3$) converge to the optimal values
  as expected.
  In this case,
  the initial weights deviated far from the optimal ones,
  and the rates of change of the weights are highest between $i=0$ and $1$.
  From Figs. \ref{fig:sim1}(a) and (b), one can see that
  these initial characteristics cause the rapid changes of the principal axes of
  the ellipsoidal curve $V^{\mu_1}(x) = c_1^*$,
  making the next region $\Omega_2$ rather conservative.
  On the contrary,
  as shown in Fig. \ref{fig:sim1}(b),
  the invariant admissible region $\Omega_i$ becomes stationary and converges
  as the weights $w_j$'s converge to the optimal ones.
  Here, the region $\Omega_i$ can be enlarged
  by providing the larger initial admissible domain $\Omega_0$,
  or using the method \eqref{eq:alpha star equation} in Remark \ref{remark:4}
  with a feasible larger subset $\Upsilon_i$,
  or
  making the initial weights close to the optimal ones.

  The simulation results for zero initial weights
  $w_{1,0} = w_{2,0} = w_{3,0} = 0$ are given in Fig. \ref{fig:sim2}.
  Compared to the previous nonzero case,
  the initial weights were set close to the optimal ones,
  and the deviations of the weights are relatively small (Fig. \ref{fig:sim2}(b)).
  These aspects result in
  the limit set $\hat \Omega$ shown in Fig. \ref{fig:sim2}(a),
  being larger and less conservative
  than the limit set $\hat \Omega$ in Fig. \ref{fig:sim1}(a).
  As can be seen from Fig. \ref{fig:sim2}(a),
  there is no significant change in the principal axes,
  making $\hat \Omega$ approximately equal to $\Omega_1$.
  While the existing PI generates $\mu_i$ over the whole domain $\mathcal{D}$,
  which is time consuming and does not guarantee the admissibility on $\mathcal{D}$,
  the proposed PI generates $\mu_i$ and the region $\Omega_i$,
  on which $\mu_i$ is invariantly admissible.
  As discussed earlier,
  the existing PI does not guarantee the admissibility on the whole domain $\mathcal{D}$
  unless the aforementioned strict assumptions regarding feasibility and stability are imposed.


\section{Conclusions}

  This paper precisely defined an invariantly admissible policy,
  the refined notion of an admissible policy in terms of feasibility,
  closed-loop Lyapunov stability, and invariance.
  Then, as a generalization of the existing SAM \cite{Beard1995,Beard1997},
  the invariantly admissible PI method was proposed
  that has the general update rule of the next region for invariant admissibility.
  The update rule for the next compact region based on the current value function was also proposed,
  and under this update rule,
  we mathematically showed the invariant admissibility
  of the generated policies and regions $(\mu_i, \Omega_{i})$;
  the monotonic decreasing property and
  uniform convergence
  of the sequence of corresponding value functions
  were also presented under certain conditions.
  Unlike the existing SAM \cite{Beard1995,Beard1997},
  the proposed PI method and the update rule did not implicitly assume
  the feasibility and the closed-loop stability \emph{on the Lyapunov domain}
  while the algorithm runs.
  Finally, numerical simulations were provided
  to illustrate the proposed PI method and its effectiveness.



\bibliographystyle{elsarticle-num}
\bibliography{AdmissibilityOfPI}

\begin{thebibliography}{10}
\expandafter\ifx\csname url\endcsname\relax
  \def\url#1{\texttt{#1}}\fi
\expandafter\ifx\csname urlprefix\endcsname\relax\def\urlprefix{URL }\fi
\expandafter\ifx\csname href\endcsname\relax
  \def\href#1#2{#2} \def\path#1{#1}\fi

\bibitem{Bellman1957}
R.~Bellman, Dynamic Programming, NJ: Princeton Univ., 1957.

\bibitem{Kirk2004}
D.~E. Kirk, Optimal control theory: an introduction, Dover Pubns, 2004.

\bibitem{Lewis1995}
F.~L. Lewis, V.~L. Syrmos, Optimal control, John Wiley, 1995.

\bibitem{Abu-Khalaf2005}
M.~Abu-Khalaf, F.~L. Lewis, Nearly optimal control laws for nonlinear systems
  with saturating actuators using a neural network \textrm{HJB} approach,
  Automatica 41~(5) (2005) 779--791.

\bibitem{Adhyaru2011}
D.~M. Adhyaru, I.~N. Kar, M.~Gopal, Bounded robust control of nonlinear systems
  using neural network-based \textrm{HJB} solution, Neural Comput. Appl. 20~(1)
  (2011) 91--103.

\bibitem{Alwardi2012}
H.~Alwardi, S.~Wang, L.~S. Jennings, An adaptive domain decomposition method
  for the \textrm{H}amilton-\textrm{J}acobi-\textrm{B}ellman equation, J. Glob.
  Optim. (2012) 1361--1373.

\bibitem{Beard1995}
R.~W. Beard, Improving the closed-loop performance of nonlinear systems, Ph.D.
  thesis, Rensselaer Polytechnic Institute (1995).

\bibitem{Beard1997}
R.~W. Beard, G.~N. Saridis, J.~T. Wen, Galerkin approximations of the
  generalized \textrm{H}amilton-\textrm{J}acobi-\textrm{B}ellman equation,
  Automatica 33~(12) (1997) 2159--2177.

\bibitem{Cheng2007}
T.~Cheng, F.~L. Lewis, M.~Abu-Khalaf, Fixed-final-time-constrained optimal
  control of nonlinear systems using neural network \textrm{HJB} approach, IEEE
  Trans. Neural. Netw. 18~(6) (2007) 1725--1737.

\bibitem{Lee2012_WCCI}
J.~Y. Lee, J.~B. Park, Y.~H. Choi, Integral reinforcement learning with
  explorations for continuous-time nonlinear systems, in: Proc. Int. Jt. Conf.
  Neural Netw. (IJCNN), 2012, pp. 1042--1047.

\bibitem{Lewis2009}
F.~L. Lewis, D.~Vrabie, Reinforcement learning and adaptive dynamic programming
  for feedback control, IEEE Trans. Circuits Syst. Mag. 9~(3) (2009) 32--50.

\bibitem{Leake1967}
R.~J. Leake, R.-W. Liu, Construction of suboptimal control sequences, SIAM J.
  Control 5~(1) (1967) 54--63.

\bibitem{Munos1999}
R.~Munos, L.~C. Baird, A.~W. Moore, Gradient descent approaches to
  neural-net-based solutions of the
  \textrm{H}amilton-\textrm{J}acobi-\textrm{B}ellman equation, in: Proc. Int.
  Jt. Conf. Neural Netw. (IJCNN), Vol.~3, 1999, pp. 2152--2157.

\bibitem{Murray2002}
J.~J. Murray, C.~J. Cox, G.~G. Lendaris, R.~Saeks, Adaptive dynamic
  programming, IEEE Trans. Syst. Man. Cybern. C Appl. Rev. 32~(2) (2002)
  140--153.

\bibitem{Saridis1979}
G.~N. Saridis, C.~S.~G. Lee, An approximation theory of optimal control for
  trainable manipulators, IEEE Trans. Syst. Man Cybern. 9~(3) (1979) 152--159.

\bibitem{Spencer1996}
B.~F. Spencer~Jr., T.~L. Timlin, M.~K. Sain, S.~J. Dyke, Series solution of a
  class of nonlinear optimal regulators, J. Optim. Theory Appl. 91~(2) (1996)
  321--345.

\bibitem{Vrabie2009b}
D.~Vrabie, F.~L. Lewis, Neural network approach to continuous-time direct
  adaptive optimal control for partially unknown nonlinear systems, Neural.
  Netw. 22~(3) (2009) 237--246.

\bibitem{Wang2003}
S.~Wang, L.~S. Jennings, K.~L. Teo, Numerical solution of
  \textrm{H}amilton-\textrm{J}acobi-\textrm{B}ellman equations by an upwind
  finite volume method, J. Glob. Optim. 27~(2-3) (2003) 177--192.

\bibitem{Lee2012}
J.~Y. Lee, J.~B. Park, Y.~H. Choi, Integral $\mathit{Q}$-learning and
  explorized policy iteration for adaptive optimal control of continuous-time
  linear systems, Automatica 48~(11).

\bibitem{Murray2003}
J.~J. Murray, C.~J. Cox, R.~E. Saeks, The adaptive dynamic programming theorem,
  in: Stability and Control of Dynamical Systems with Applications, Springer,
  2003, pp. 379--394.

\bibitem{Kleinman1968}
D.~Kleinman, On an iterative technique for \textrm{R}iccati equation
  computations, IEEE Trans. Automat. Contr. 13~(1) (1968) 114--115.

\bibitem{Khalil2002}
H.~K. Khalil, Nonlinear systems, Prentice Hall, 2002.

\bibitem{Knobloch1974}
H.~Knobloch, F.~Kappel, Gew{\"o}hnliche Differentialgleichungen, BG Teubner,
  1974.

\bibitem{Kaslik2005}
E.~Kaslik, A.~M. Balint, S.~Balint, Methods for determination and approximation
  of the domain of attraction, Nonlinear Anal. Theory Methods Appl. 60~(4)
  (2005) 703--717.

\bibitem{Marsden1993}
J.~E. Marsden, Elementary classical analysis, Macmillan, 1993.

\bibitem{Nevistic1996}
V.~Nevisti{\'c}, J.~A. Primbs, Constrained nonlinear optimal control: a
  converse \textrm{HJB} approach, Technical report 96-021.

\end{thebibliography}





\end{document}